\documentclass{amsart}
\usepackage{color}
\usepackage[T1]{fontenc}
\usepackage{amsfonts}
\usepackage{amssymb}
\usepackage{amsmath}
\usepackage{amsthm}
\usepackage{amscd}
\usepackage{pgf}
\usepackage{tikz}
\usetikzlibrary{matrix}
\newtheorem{theorem}{Theorem}[section]
\newtheorem{lemma}[theorem]{Lemma}
\newtheorem{proposition}[theorem]{Proposition}
\newtheorem{corollary}[theorem]{Corollary}
\newtheorem{definition}[theorem]{Definition}

\newtheorem{remark}{Remark}
\def\a{\alpha} 
\def\lm{\lambda} 
\def\vp{\varphi} 
\newcommand{\Q}{{\mathbf Q}}
\newcommand{\FF}{\mathcal H}
\newcommand{\Prob}{\mathfrak M}
\newcommand{\h}[1]{\mathcal{H}_{#1}} 
\providecommand{\pp}{\le}      
\renewcommand{\textgreater}{\rangle}  
\providecommand{\pp}{\le}      
\newcommand{\id}{\operatorname{Id}}
\newcommand{\f}[1]{\mathcal{F}_{#1}}   
\newcommand{\g}[1]{\mathcal{G}_{#1}} 
\newcommand{\dom}[1]{\operatorname{Dom}{#1}} 
\newcommand{\var}{\operatorname{var}}
\newcommand{\cov}{\operatorname{cov}}
\newcommand{\car}{\mathbf 1}
\def\AtEA{A\times E_A}
\def\P{\mathbf P}
\def\I{\mathcal{J}} 
\def\esp#1{\mathbf E\left[#1\right]}
\def\espp#1{\mathbf E^\prime\left[#1\right]}
\def\cyl{\mathcal S}
\def\exv{\mathcal G}
\def\trace{\operatorname{trace}}
\def\dif{\operatorname{ d}\!}
\def\DD{{\mathbf D}}

\def\a{\alpha} 
\newcommand{\Lip}{\operatorname{Lip}}
\newcommand{\HC}{\operatorname{H-C}^1}
\newcommand{\bY}{{\mathbb Y}}
\newcommand{\fN}{{\mathfrak N}}
\newcommand{\bM}{{\mathbf M}}
\newcommand{\R}{{\mathbf R}}
\newcommand{\N}{{\mathbf N}}
\def\F{\mathcal F}
\newcommand{\TVlip}{{\text{\textsc{TV}}-\operatorname{Lip}}}

\begin{document}
\title{Malliavin and Dirichlet structures for independent random variables}
\author{L. Decreusefond \and H. Halconruy}




%


%
\begin{abstract}
  On any denumerable product of probability spaces, we construct a Malliavin
  gradient and then a divergence and a number operator. This yields a Dirichlet
  structure which can be shown to approach the usual structures for Poisson and
  Brownian processes. We obtain versions of almost all the classical functional
  inequalities in discrete settings which show that the Efron-Stein inequality
  can be interpreted as a Poincar\'e inequality or that the Hoeffding decomposition
  of $U$-statistics can be interpreted as an avatar of the Clark representation formula. Thanks to our
  framework, we obtain a bound for the distance between the distribution of any
  functional of independent variables and the Gaussian  and Gamma distributions.
\end{abstract}
\maketitle{}
\section{Introduction}
\label{sec:introduction}
There are two motivations to the present paper. After some years of development,
the Malliavin calculus has reached a certain maturity. The most complete
theories are for Gaussian processes (see for instance
\cite{nualart.book,ustunel2000}) and Poisson point processes (see for instance
\cite{MR99d:58179,MR2531026}). When looking deeply at the main proofs, it
becomes clear that the independence of increments plays a major role in the
effectiveness of the concepts. At a very formal level, independence and
stationarity of increments induce the martingale representation property which
by induction entails the chaos decomposition, which is one way to develop
Malliavin calculus for Poisson~\cite{nualart88_1}, L\'evy
processes~\cite{NUALART2000109} and Brownian motion. It thus motivates to
investigate the simplest situation of all with independence: That of a family of
independent, non necessarily identically distributed, random variables.

The second motivation comes from Stein's method\footnote{Giving an
  exhaustive bibliography about Stein's method is somehow impossible (actually,
  MathSciNet refers more than 500 papers on this subject). The references given
  here are only entry points to the items alluded to.}. The Stein method which
was initially developed to quantify the rate of convergence in the Central
Limit Theorem~\cite{stein1972} and then for Poisson convergence
\cite{MR0370693}, can be decomposed in three steps (see
\cite{decreusefond_stein-dirichlet-malliavin_2015}). In the first step, we have
to find a functional identity which characterizes the target distribution and
solve implicitly or explicitly (as in the semi-group method) the so-called
Stein's equation. It reduces the computation of the distance to the calculation
of
\begin{equation*}
  \sup_{F\in \mathcal F}\Bigl(  \esp{L_1F(X)}+\esp{L_2F(X)} \Bigr),
\end{equation*}
where $\mathcal F$ is the class of functions solutions of the Stein equation, which is related to the set of test functions $\FF$ induced by the distance we are considering, $L_1$ and $L_2$ are two functional operators and $X$ is a
random variable whose distribution we want to compare to the target
distribution. For instance, if the target distribution is the Gaussian law on
$\R$,
\begin{equation*}
  L_1F(x)=xF'(x) \text{ and } L_2F(x)=-F''(x).
\end{equation*}
If the target distribution is the Poisson law of parameter $\lambda$,
\begin{equation*}
  L_1F(n)=n\,(F(n)-F(n-1)) \text{ and } L_2F(n)=\lambda(F(n+1)-F(n)).
\end{equation*}
In the next step, we have to take into account how $X$ is defined and transform
$L_1F$ such that it can be written as $-L_2F+\text{remainder}$. This remainder
is what gives the rate of convergence. To make the transformation of $L_1F$,
several approaches appeared along the years. One of the most popular approach
(see for instance~\cite{barbour_introduction}) is to use exchangeable pairs:
Construct a copy $X'$ of $X$ with good properties which gives another expression
of $L_1F$, suitable to a comparison with $L_2F$. To be more specific, for the
proof of the CLT, it is necessary to create an exchangeable pair $(S,S')$ with
$S=\sum_{i=1}^nX_i$. This is usually done by first, choosing uniformly an index
$I\in \{1,\cdots,n\}$ and then, replacing $X_I$ with $X'$ an independent copy of
$X_I$, so that the couple
\begin{math}
  ( S,\ S'=S-X_I+X')
\end{math}
is an exchangeable pair. This means that
\begin{equation}\label{eq_gradient:6}
  \esp{F(S')\,|\, I=a;\  X_b,\, b\not = a}=\esp{F(S)\, |\, X_b,\, b\not = a}.
\end{equation}
Actually, it is the right-hand-side of~\eqref{eq_gradient:6} which gave us some
clue on how to proceed when dealing with functionals more general than the sum
of random variables. An alternative to exchangeable pairs, is the
size-biased~\cite{Chen2011} or zero biased~\cite{MR1484792} couplings, which
again conveniently transform $L_1F$. For Gaussian approximation, it amounts to
find a distribution $X^*$ such that
\begin{equation*}
  \esp{L_1F(X)}=\esp{F''(X^*)}.
\end{equation*}
Note that for $S$ as above, one can choose $S^*=S'$. If the distribution of
$X^*$ is absolutely continuous with respect to that of $X$, with Radon
derivative $\Lambda$, we obtain
\begin{equation*}
  \esp{L_1F(X)}=\esp{F''(X)\,\Lambda(X)},
\end{equation*}
which means that we are reduced to estimate how far $\Lambda$ is from the
constant random variable equal to $1$. This kind of identity, where the second
order derivative is multiplied by a weight factor, is reminiscent to what can
be obtained via integration by parts. Actually, Nourdin and Peccati (see
\cite{Nourdin:2012fk}) showed that the transformation step can be advantageously
made simple using integration by parts in the sense of Malliavin calculus. This
works well only if there exists a Malliavin gradient on the space on which $X$
is defined (see for instance~\cite{DST:functional}). That is to say, that up to
now, this approach is restricted to functionals of Rademacher
\cite{Nourdin2010}, Poisson~\cite{DST:functional,taqqu} or Gaussian random
variables~\cite{MR2118863} or processes~\cite{CD:2012,CD:2014}. Then, strangely
enough, the first example of applications of the Stein's method which was the
CLT, cannot be handled through this approach. On the one hand, exchangeable
pairs or size-biased coupling have the main drawback to have to be adapted to
each particular version of $X$. On the other hand, Malliavin integration by
parts are in some sense more automatic but we need to be provided with a
Malliavin
structure. 

The closest situation to our investigations is that of the Rademacher space,
namely $\{-1,1\}^\N$, equipped with the product probability $\otimes_{k\in \N}
\mu_k$ where $\mu_k$ is a Bernoulli probability on $\{-1,1\}$.

The gradient on the Ra\-de\-ma\-cher space (see~\cite{Nourdin2010,MR2531026}) is
usually defined~as
\begin{multline}
  \label{eq_gradient_spa_v2:2}
  \hat{D}_kF(X_1,\cdots,X_n)=\esp{X_k\,F(X_1,\cdots,X_n)\, |\, X_l,\, l\neq k}\\
  \shoveleft{=\P(X_{k}=1)\,F(X_{1},\cdots,+1,\cdots, X_{n})}\\
  -\P(X_{k}=-1)\, F(X_{1},\cdots,-1,\cdots, X_{n}) ,
\end{multline}
where the $\pm 1$ are put in the $k$-th coordinate. It requires, for its very
definition to be meaningful, either that the random variables are real valued or that
they only have two possible outcomes. In what follows, it must be made clear
that all the random variables may leave on different spaces, which are only
supposed to be Polish spaces. That means that in the definition of the gradient,
we cannot use any algebraic property of the underlying spaces. Some of our
applications does concern random variables with finite number of outcomes but it
does not seem straightforward to devise what should be the weights, replacing
$\P(X_{k}=1)$ and $-\P(X_{k}=-1)$. Furthermore, many applications, notably those
revolving around functional identities, rely not directly on the gradient $D$
but rather on the operator number $L=-\delta D$ where $\delta$ is the adjoint,
in a sense to be defined later. It turns out that for the Rademacher space, the
operators $\hat{L}=-\hat{\delta}\hat{D}$ defined according to
\eqref{eq_gradient_spa_v2:2} and $L$ defined in Definition~\ref{def:gradient} do
coincide. Our framework then fully generalizes what is known about Rademacher
spaces.

Since Malliavin calculus is agnostic to any time reference, we do not even
assume that we have an order on the product space. It is not a major feature
since a denumerable $A$ is by definition in bijection with the set of natural
integers and thus inherits of at least one order structure. However, this added
degree of freedom appears to be useful (see the Clark decomposition of the
number of fixed points of a random permutations in Section~\ref{sec:appl-perm})
and bears strong resemblance with the different filtrations which can be put on
an abstract Wiener space, via the notion of resolution of the
identity~\cite{MR1428114}. During the preparation of this work, we found strong
reminiscences of our gradient with the map $\Delta$, introduced
in~\cite{boucheron_concentration_2013,MR822716} for the proof of the Efron-Stein
inequality, defined by
\begin{equation*}
  \Delta_k F(X_1,\cdots,X_n)=\esp{F\, |\, X_1,\cdots,X_k}- \esp{F\, |\, X_1,\cdots,X_{k-1}}.
\end{equation*}
Actually, our point of view diverges from that of these works as we do not focus
on a particular inequality but rather on the intrinsic properties of our newly
defined gradient.

We would like to stress the fact that our Malliavin-Dirichlet structure gives a
unified framework for many results scattered in the literature. We hope to give
new insights on why these apparently disjointed results (Efron-Stein,
exchangeable pairs, etc.) are in fact multiple sides of the same coin.

We proceed as follows. In Section~\ref{sec:mall-calc-indep}, we define the
gradient $D$ and its adjoint $\delta$, which we call divergence as it appears as
the sum of \textsl{the partial derivatives}, as in~$\R^n$. We establish a Clark
representation formula of square integrable random variables and an Helmholtz
decomposition of vector fields. Clark formula appears to reduce to the Hoeffding
decomposition of $U$-statistics when applied to such functionals. We establish a
log-Sobolev inequality, strongly reminding that obtained for Poisson
processes~\cite{Wu:2000lr}, together with a concentration inequality. Then, we
define the number operator $L=\delta D$. It is the generator of a Markov process
whose stationary distribution is the tensor probability we started with. We show
in Section~\ref{sec:dirichlet-structures} that we can retrieve the classical
Dirichlet-Malliavin structures for Poisson processes and Brownian motion as
limits of our structures. We borrow for that the idea of convergence of
Dirichlet structures to~\cite{bouleau_theoreme_2005}.
The construction of random permutations in~\cite{Kerov2004a}, which is similar
in spirit to the so-called Feller coupling (see \cite{MR1177897}), is an
interesting situation to apply our results since this construction involves a
cartesian product of distinct finite spaces. In Section~\ref{sec:appl-perm}, we
present several applications of our results. In subsection~\ref{sec:representations}, we
derive the chaos decomposition of the number of fixed points of a random
permutations under the Ewens distribution. This yields an exact expression for
the variance of this random variable.  To the price of an additional
complexity, it is certainly possible to find such a decomposition for the number
of $k$-cycles in a random permutation. In subection~\ref{sec:stein}, we give an
 analog to Theorem 3.1 of \cite{MR2520122,taqqu}, which is a general bound of the
 Kolmogorov Rubinstein distance to a Gaussian or Gamma distribution, in terms of
 our gradient $D$. We apply this to a degenerate U-statistics of order~$2$.

\section{Malliavin calculus for independent random variables}
\label{sec:mall-calc-indep}
Let $A$ be an at most denumerable set equipped with the counting measure:
\begin{equation*}
  L^{2}(A)=\left\{u\, :\,A\to \R,\ \sum_{a\in A}|u_{a}|^{2}<\infty \right\} \text{ and } \langle {u,v} \rangle_{L^{2}(A)}=\sum_{a\in A}u_{a}v_{a}.
\end{equation*}
Let ($E_a,a\in A$) be a family of Polish spaces. For
any $a\in A$, let $\mathcal{E}_a$ and $\P_a$ be respectively a $\sigma$-field
and a probability measure defined on $E_a$. We consider the probability space
$E_A=\prod_{a\in A} E_a$ equipped with the product $\sigma$-field $\mathcal
E_{A}=\underset{{a\in A}}\vee \mathcal E_a$ and the tensor product measure
$\P=\underset{{a\in A}}\otimes\P_a$. \\
The coordinate random variables are denoted
by $(X_a, a\in A)$. For any $B\subset A$, $X_B$ denotes the random vector $(X_a,
a\in B)$, defined on $E_B=\prod_{a\in B} E_a$ equipped with the probability
$\P_B=\underset{{a\in B}}\otimes\P_a$. \\
A process $U$ is a measurable random
variable defined on $(\AtEA,\, \mathcal P(A)\otimes \mathcal E_A)$.\\
 We denote by
$L^2(\AtEA)$ the Hilbert space of processes which are square integrable with
respect to the measure $\sum_{a\in A}\varepsilon_a\otimes \P_A$ (where
$\varepsilon_a$ is the Dirac measure at point $a$):
\begin{equation*}
  \|U\|_{L^2(\AtEA)}^2=\sum_{a\in A}\esp{U_a^2} \text{ and }    \langle U,\, V\rangle_{L^2(\AtEA)} = \sum_{a\in A}\esp{U_aV_a}.
\end{equation*}
Our presentation follows closely the usual construction of Malliavin calculus.
\begin{definition}
  A random variable $F$ is said to be cylindrical if there exist a finite subset
  $B\subset A$ and a function $F_B:E_B\longrightarrow L^2(E_A)$ such that
  \begin{math}
    F=F_B\circ r_B,
  \end{math}
  where $r_B$ is the restriction operator:
  \begin{align*}
    r_B\, :\,  
    E_A&\longrightarrow E_B\\ 
    (x_a,a\in A) &\longmapsto (x_a,a\in B).
  \end{align*}
  This means that $F$ only depends on the finite set of random variables
  $(X_a,\, a\in B)$.

  It is clear that $\cyl$ is dense in $L^2(E_A)$.
\end{definition}

The very first tool to be considered is the discrete gradient, whose form has
been motivated in the introduction.

We first define the gradient of cylindrical functionals, for there is no
question of integrability and then extend the domain of the gradient to a larger
set of functionals by a limiting procedure. In functional analysis terminology,
we need to verify the closability of the gradient: If a sequence of functionals
converges to $0$ and the  sequence of their gradients  converges, then it should also converges to
$0$. This is the only way to guarantee in the limiting procedure that the limit
does not depend on the chosen sequence.
\begin{definition}[Discrete gradient]
  \label{def:gradient}
  For $F\in \cyl$, $DF$ is the process of $L^2(A\times E_A)$ defined by one of
  the following equivalent formulations: For all $a\in A$,
  \begin{align*}
    D_aF(X_A)& =F(X_A)-\esp{F(X_{A})\, |\, \exv_a}\\
             &= F(X_A)-\int_{E_a}F(X_{A\smallsetminus{a}},x_a)\dif\P_a(x_a)\\
             &=F(X_{A})-\espp{F(X_{A\smallsetminus{a}},X_a')}
  \end{align*}
  where $X'_a$ is an independent copy of $X_a$.
\end{definition}
\begin{remark}
  A straightforward calculation shows that for any $F,G \in \cyl$, any $a\in A$,
  we have
  \begin{equation*}
   D_{a}(FG)=F\, D_{a}G + G\, D_{a}F - D_{a}F\, D_{a}G -\esp{FG\, |\, \exv_{a}}+\esp{F\, |\, \exv_{a}}\esp{G\, |\, \exv_{a}}.
  \end{equation*}
  This formula has to be compared with the formula $D(FG)=F\, DG+G\, DF$ for the
  Gaussian Malliavin gradient (see \eqref{eq_gradient_spa_v2:19} below) and $D(FG)=F\, DG+G\, DF+DF\, DG$ for the Poisson
  gradient (see \eqref{eq_gradient_spa_v2:18} below).
\end{remark}

For $F\in \cyl$, there exists a finite subset $B\subset A$ such that $F=F_B\circ
r_B$. Thus, for every $a\notin B$, $F$ is $\mathcal{G}_a$-measurable and then
$D_aF=0$. This implies that
\begin{equation*}
  \|DF\|_{L^2(A\times E_A)}^2
  =\esp{\sum_{a\in A}|D_aF|^2}
  =\esp{\sum_{a\in B}|D_aF|^2}
  <\infty,
\end{equation*}
hence $(D_aF, a\in A)$ defines an element of $L^2(\AtEA)$.
\begin{definition}
  The set of simple processes, denoted by $\cyl_0(l^2(A))$ is the set of random
  variables defined on $\AtEA$ of the form
  \begin{equation*}
    U=\sum_{a\in B} U_a \, \car_a,
  \end{equation*}
  for $B$ a finite subset of $A$ and such that $U_a$ belongs to $\cyl$ for any
  $a\in B$.
\end{definition}
The key formula for the sequel is the so-called integration by parts. It amounts
to compute the adjoint of $D$ in $L^{2}(\AtEA)$.
\begin{theorem}[Integration by parts]
  \label{thm:ipp}
  Let $F\in\mathcal{S}$. For every simple process~$U$,
  \begin{equation}\label{IPP}
    \langle DF, U\rangle_{L^2(\AtEA)}= \esp{F\ \sum_{a\in A}D_aU_a}. 
  \end{equation}
\end{theorem}

%
%
%
Thanks to the latter formula, we are now in position to prove the closability
of~$D$: For $(F_{n},\, n\ge 1)$  a sequence of cylindrical functionals, 
\begin{equation*}
  \left(  F_{n}\xrightarrow[L^{2}(E_{A})]{n\to \infty}0\,\text{ and }\, DF_{n}\xrightarrow[L^{2}(\AtEA)]{n\to \infty}\eta \right) \Longrightarrow \eta=0.
\end{equation*}
\begin{corollary}\label{closability}
  The operator $D$ is closable from $L^2(E_A)$ into $L^2(\AtEA)$.
\end{corollary}

We denote the domain of $D$ in $L^2(E_A)$ by $\DD$, the closure of the class of
cylindrical functions with respect to the norm
\begin{equation*}
  \|F\|_{1,2}=\left(\|F\|_{L^2(E_A)}^2+\|DF\|_{L^2(\AtEA)}^2\right)^{\frac{1}{2}}.
\end{equation*}
We could as well define $p$-norms corresponding to $L^p$ integrability. However,
for the current applications, the case $p=2$ is sufficient and the apparent lack
of hypercontractivity of the Ornstein-Ulhenbeck semi-group (see below
Section~\ref{sec:ornst-uhlenb-semi}) lessens the probable usage of other
integrability order.

Since $\DD$ is defined as a closure, it is often useful to have a general
criterion to ensure that a functional $F$, which is not cylindrical, belongs to
$\DD$. The following criterion exists as is in the settings of Wiener and
Poisson spaces.
\begin{lemma}
  \label{lem:boundedness} If there exists a sequence $(F_n,\, n\ge 1)$ of
  elements of $\DD$ such that
  \begin{enumerate}
  	\item $F_n$ converges to $F$ in $L^2(E_A)$,
  	\item $\sup_n \|DF_n\|_{\DD}$ is finite,
  \end{enumerate}
   then $F$ belongs to $\DD$ and
  $DF=\lim_{n\to \infty }DF_n$ in $\DD$.
\end{lemma}

\subsection{Divergence}
We can now introduce the adjoint of $D$, often called the divergence as for the
Lebesgue measure on $\R^{n}$, the usual divergence is the adjoint of the usual
gradient.
\begin{definition}[Divergence]
  Let
  \begin{multline*}
    \dom{\delta}=\Bigl\{U\in L^2(\AtEA):\\ \exists\, c>0,\, \forall\, F\in\DD,\
    |\langle DF, U\rangle_{L^2(\AtEA)}|\le c\,\|F\|_{L^2(E_A)}\Bigr\}.
  \end{multline*}
  For any $U$ belonging to $\dom\delta$, $\delta U$ is the element of $L^2(E_A)$
  characterized by the following identity
  \begin{equation*}
    \langle DF, U  \rangle_{L^2(\AtEA)}=\esp{F\,\delta U}, \text{ for all } F\in\DD.
  \end{equation*}  
  The integration by parts formula~\eqref{IPP} entails that for every
  $U\in\dom{\delta}$,
  \begin{equation*}
    \delta U=\sum_{a\in A} D_aU_a.
  \end{equation*}
\end{definition}
%
%
In the setting of Malliavin calculus for Brownian motion, the divergence of
adapted processes coincide with the It\^o integral and the square moment of
$\delta U$ is then given by the It\^o isometry formula. We now see how this
extends to our situation.
\begin{definition}
  The Hilbert space $\DD(l^2(A))$ is the closure of $\mathcal S_0(l^2(A))$ with
  respect to the norm
  \begin{equation*}
    \|U\|_{\DD(l^2(A))}^2=\esp{\sum_{a\in A}|U_a|^2}+\esp{\sum_{a\in A}\sum_{b\in A}|D_aU_b|^2}.
  \end{equation*}
\end{definition}
In particular, this means that the map $DU=(D_aU_b, \ a,b\in A)$ is
Hilbert-Schmidt as a map from $L^2(\AtEA)$ into itself. As a consequence, for
two such maps $DU$ and $DV$, the map $DU\circ DV$ is trace-class (see
\cite{MR1336382}) with
\begin{equation*}
  \trace(DU \circ DV)=\sum_{a,b\in A} D_aU_b\  D_bV_a.
\end{equation*}
The next formula is the counterpart of the It\^o isometry formula for the
Brownian motion, sometimes called the Weitzenb\"ock formula (see \cite[Eqn.
(4.3.3)]{MR2531026}) in the Poisson settings.
\begin{theorem}\label{Ddelta}
  The space $\DD(l^2(A))$ is included in $\dom\delta$. For any $U,\, V$
  belonging to $\DD(l^2(A))$,
  \begin{equation}\label{norm_delta_1}
    \esp{\delta U\ \delta V}=\esp{\trace(DU\circ DV)}.
  \end{equation}
\end{theorem}

\begin{remark}
  It must be noted that compared to the analog identity for the Brownian and the
  Poisson settings, the present formula is slightly different. For both
  processes, with corresponding notations, we have
  \begin{equation*}
    \|\delta U\|_{L^{2}}^{2}=\|U\|_{L^{2}}^{2}+\trace(DU\circ DV).
  \end{equation*}
  The absence of the term $\|U\|_{L^{2}}^{2}$ gives to our formula a much
  stronger resemblance to the analog equation for the Lebesgue measure. As in
  this latter case, we do have here $\delta \car=0$ whereas for the Brownian
  motion, it yields the It\^o integral of the constant function equal to one.

  If $A=\N$, let $\F_{n}=\sigma\{X_{k},\, k\le n\}$ and assume that $U$ is
  adapted, i.e. for all $n\ge 1$, $U_{n}\in \F_{n}$. Then, $D_{n}U_{k}=0$ as
  soon as $n>k$, hence
  \begin{equation*}
    \esp{\delta U^{2}}=\esp{\sum_{n=1}^{\infty}\Bigl(U_{n}-\esp{U_{n}\,|\, \F_{n-1}}\Bigr)^{2}},
  \end{equation*}
  i.e. $ \esp{\delta U^{2}}$ is the $L^{2}(\N\times E_{\N})$-norm of the
  innovation process associated to $U$, which appears in filtering theory.
\end{remark}
%

%
%
%
%
%
%

\subsection{Ornstein-Uhlenbeck semi-group and generator}
\label{sec:ornst-uhlenb-semi}
Having defined a gradient and a divergence, one may consider the Laplacian-like
operator defined by $L=-\delta D$, which is also called the number operator in
the settings of Gaussian Malliavin calculus.
\begin{definition}
  \label{def_Article-part1:1}
  The number operator, denoted by $L$, is defined on its domain
  \begin{equation*}
    \dom L=\left\{F\in L^2(E_A) : \esp{\displaystyle\sum_{a\in A}|D_aF|^2}<\infty\right\}
  \end{equation*} 
  by
  \begin{equation}\label{eq_gradient_spa_v2:1}
    LF=-\delta DF=-\displaystyle\sum_{a\in A} D_aF.
  \end{equation}
\end{definition}

The map $L$ can be viewed as the generator of a symmetric Markov process $X$,
which is ergodic, whose stationary probability is $\P_A$. Assume first that $A$
is finite. Consider $(Z(t),t\ge 0)$ a Poisson process on the half-line of rate
$|A|$, and the process $X(t)=(X_1(t),\cdots,X_N(t),\, t\ge 0)$ which evolves
according to the following rule: At a jump time of $Z$,
\begin{itemize}
\item Choose randomly (with equiprobability) an index $a\in
  A$, 
\item Replace $X_a$ by an independent random variable $X_a'$ distributed
  according to $\P_a$.
\end{itemize}
For every $x\in E_A$, $a\in A$, set $x_{\neg
  a}=(x_1,\cdots,\,x_{a-1},\,x_{a+1},\,\cdots,\,x_{|A|})$. The generator of the
Markov process $X$ is clearly given by
\begin{equation*}
  |A| \ \sum_{a\in A}\frac{1}{|A|} \,\int_{E_a}\Bigl(F(x_{\neg a},\, x_a')-F(x)\Bigr) \dif\P_a(x_a')
  =-\sum_{a\in A} D_aF(x).
\end{equation*}
The factor $|A|$ is due to the intensity of the Poisson process $Z$ which jumps
at rate $|A|$, the factor $|A|^{-1}$ is due to the uniform random choice of an
index $a\in A$. Thus, for a finite set $A$, $L$ coincides with the generator of
$X$. If we denote by $P=(P_{t},t\ge 0)$ the semi-group of $X$: For any $x\in
E_{A}$, for any bounded $f\, :\, E_{A}\to \R$,
\begin{equation*}
  P_{t}f(x)=\esp{f(X(t))\, |\, X(0)=x}.
\end{equation*}
Then, $(P_{t},t\ge 0)$ is a strong Feller semi-group on $L^{\infty}(E_{A})$.
This result still holds when $E_{A}$ is denumerable.
\begin{theorem}
  \label{thm:PtADenumerable}
  For any denumerable set $A$, $L$ defined as in \eqref{eq_gradient_spa_v2:1}
  generates a strong Feller continuous semi-group $(P_{t},t\ge 0)$ on
  $L^{\infty}(E_{A})$.

  As a consequence, there exists a Markov process $X$ whose generator is $L$ as
  defined in \eqref{eq_gradient_spa_v2:1}. It admits as a core (a dense subset
  of its domain) the set of cylindrical functions.
\end{theorem}
From the sample-path construction of $X$, the next result is straightforward for
$A$ finite and can be obtained by a limiting procedure for $A$ denumerable.
\begin{theorem}[Mehler formula]
  For $a\in A$, $x_a\in E_{A}$ and $t>0$, let $X_a(x_a,t)$ the random variable
  defined by
  \begin{equation*}
    X_a(x_a,t)=
    \begin{cases}
      x_a & \text{ with probability } (1-e^{-t}),\\
      X'_a & \text{ with probability } e^{-t},
    \end{cases}
  \end{equation*}
  where $X'_a$ is a $\P_a$-distributed random variable independent from
  everything else. In other words, if $P_a^{x_a,t}$ denotes the distribution of
  $ X_a(x_a,t)$, $P_a^{x_a,t}$ is a convex combination of $\varepsilon_{x_a}$
  and $\P_a$:
  \begin{equation*}
    P_a^{x_a,t}=(1-e^{-t})\, \varepsilon_{x_a} + e^{-t}\, \P_a.
  \end{equation*}
  For any $x\in E_A$, any $t>0$,
  \begin{equation*}
    P_tF(x)=\int_{E_A} F(y)\ \underset{a\in A}{\otimes}  \dif \P_a^{x_a,t}(y_a).
  \end{equation*}
  It follows easily that $(P_t,\, t\ge 0)$ is ergodic and stationary:
  \begin{equation*}
    \lim_{t\to \infty } P_tF(x)=\int_{E_A}F\dif\P
    \text{ and } X(0)\stackrel{\text{law}}{=}\P \Longrightarrow X(t)\stackrel{\text{law}}{=} \P.
  \end{equation*}
\end{theorem}
We then retrieve the classical formula (in the sense that it holds as is for
Brownian motion and Poisson process) of commutation between $D$ and the
Ornstein-Uhlenbeck semi-group.
\begin{theorem}
  \label{thm:inversionDP_t}
  Let $F\in L^2(E_A)$. For every $a\in A$, $x\in E_A$,
  \begin{equation}\label{OU-Grad}
    D_aP_tF(x)=e^{-t}P_tD_aF(x).
  \end{equation}
\end{theorem}

\section{Functional identities}
\label{sec:functional}
This section is devoted to several functional identities which constitute the
crux of the matter if we want to do some computations with our new tools.

It is classical that the notion of adaptability is linked to the support of the
gradient.
\begin{lemma}\label{Lchaos1}
  Assume that $A=\N$ and let $\f{n}=\sigma\{X_{k},\, k\le n\}$. For any
  $F\in\DD$, $F$ is $\f{k}$-measurable if and only if $D_n F=0$ for any~$n>k$.
  As a consequence, $DF=0$ if and only if $F=\esp{F}$. 
\end{lemma}
It is also well known that, in the Brownian of Poisson settings, $D$ and
conditional expectation commute.
\begin{lemma}
  \label{lem_gradient:permutation}
  For any $F\in\DD$, for any $k\ge 1$, we have
  \begin{equation}\label{chaos2}
    D_k\ \esp{F|\f{k}}=\esp{D_kF\,|\,\f{k}}.
  \end{equation}
\end{lemma}
The Brownian martingale representation theorem says that a martingale adapted to
the filtration of a Brownian motion is in fact a stochastic integral. The Clark
formula gives the expression of the integrand of this stochastic integral in
terms of the Malliavin gradient of the terminal value of the martingale. We here
have the analogous formula.
\begin{theorem}[Clark formula]\label{lchaos}
  For $A=\N$ and $F\in \DD$,
  \begin{equation*}
    F=\esp{F}+\sum_{k=1}^{\infty}D_k\,\esp{F\,|\,\f{k}}.
  \end{equation*}
  If $A$ is finite and if there is no privileged order on $A$, we can write
  \begin{equation*}
    \label{eq_gradient:4}
    F=\esp{F}+\sum_{B\subset A} \ \binom{|A|}{|B|}^{-1}\!\frac1{|B|}\sum_{b\in B} D_b \, \esp{F\,|\, X_B}.
  \end{equation*}
\end{theorem}

The chaos decomposition is usually deduced from the Clark formula by iteration.
If we apply Clark formula to $\esp{F\,|\,\f{k}}$, we get
\begin{equation*}
  D_k\esp{F\, |\, \f{k}}=\sum_{j=1}^\infty D_kD_j\esp{F \,|\, \f{j\wedge k}}=   D_k\esp{F\, |\, \f{k}},
\end{equation*}
since $j>k$ implies $D_j\esp{F\,|\, \f{k}}=0$ in view of Lemma~\ref{Lchaos1}.
Furthermore, the same holds when $k>j$ since it is easily seen that
$D_jD_k=D_kD_j$. For $j=k$, simply remark that $D_kD_k=D_k$. Hence, it seems
that we cannot go further this way to find a potential chaos decomposition.

As mentioned in the Introduction, it may be useful to reverse the time arrow.
Choose an order on $A$ so that $A$ can be seen as~$\N$. Then, let
\begin{equation*}
  \h{n}=\sigma\{X_k,\, k> n\}.
\end{equation*}
and for any $n\in\{0,\cdots,N-1\}$,
\begin{equation*}
  \h{n}^N=\h{n}\cap \f{N}\; \text{ and } \; \h{k}^N=\f{0}=\{\emptyset, \
  E_A\} \text{ for } k\ge N.
\end{equation*}
Note that $\h{0}^N=\f{N}$ and as in Lemma~\ref{Lchaos1}, $F$ is
$\h{k}$-measurable if and only if $D_n F=0$ for any $ n\le k$.

\begin{theorem}\label{lchaos:reverse}
  For every $F$ in $\DD$,
  \begin{equation*}
    F=\esp{F}+\sum_{k= 1}^{\infty}D_k\,\esp{F\,|\,\h{k-1}}.
  \end{equation*}
\end{theorem}

In the present context, the next result is a Poincar\'e type inequality as it
gives a bound for the variance of $F$ in terms of the \textsl{oscillations} of
$F$. In other context, it turns to be called the Efron-Stein inequality
\cite{boucheron_concentration_2013}. It can be noted that both the statement and
the proof are similar in the Brownian and Poisson settings.
\begin{corollary}[Poincar\'e or Efron-Stein inequality]
  \label{cor:poincare}
  For any $F\in \DD$,
  \begin{equation*}\label{poincare}
    \var(F)\le  \|DF\|_{L^2(\AtEA)}^2.
  \end{equation*}
\end{corollary}
%
%
Another corollary of the Clark formula is the following covariance identity.
\begin{theorem}[Covariance identity]
  \label{thm:cov1}
  For any $F,G\in \DD$,
  \begin{equation}\label{cov_2}
    \cov(F,G)=\esp{\sum_{k\in A}D_k\esp{F\,|\,\f{k}}\ D_kG}.
  \end{equation}
\end{theorem}

As for the other versions of the Malliavin calculus (Brownian, Poisson and
Rademacher), from (\ref{OU-Grad}), can be deduced another covariance identity.
\begin{theorem}
  \label{thm:cov2}
  For any $F,G\in \DD$,
  \begin{equation}\label{cov_2b}
    \cov(F,G)=\esp{\sum_{k\in A}D_kF\int_0^{\infty}e^{-t}P_t\esp{D_kG|\f{k}}\dif t}.
  \end{equation}
\end{theorem}
Then, using the so-called Herbst principle, we can derive a concentration
inequality, which, as usual, requires an $L^{\infty }$ bound on the derivative
of the functional to be valid.
\begin{theorem}[Concentration inequality]
  \label{thm:concentration}
  Let $F$ for which there exists an order on $A$ with
  \begin{equation*}
    M=\sup_{X\in E_{A}}\sum_{{k=1}}^{\infty}|D_{k}F(X)|\,\esp{|D_{k}F(X)|\,|\,\f{k}}<\infty.
  \end{equation*}
  Then, for any $x\ge 0$, we have
  \begin{equation*}
    \P(F-\esp{F}\ge x)\le \exp\left(-\frac{x^2}{2M}\right)\cdotp
  \end{equation*}
\end{theorem}

In the Gaussian case, the concentration inequality is deduced from the
logarithmic Sobolev inequality. This does not seem to be feasible in the present
context because $D$ is not a derivation, i.e. does not satisfy $D(FG)=F\,DG+G\,
DF$. However, we still have an LSI identity. For the proof of it, we follow
closely the proofs of \cite{privault93,Wu:2000lr}. They are based on two
ingredients: The It\^o formula and the martingale representation theorem. We get
an ersatz of the former but the latter seems inaccessible as we do not impose
the random variables to live in the same probability space and to be real
valued. Should it be the case, to the best of our knowledge, the martingale
representation formula is known only for the Rademacher space \cite[Section
15.1]{MR1155402}, which is exactly the framework of \cite{privault93}. This lack
of a predictable representation explains the conditioning in the denominator
of~\eqref{log-sob}.
\begin{theorem}[Logarithmic Sobolev inequality]
  \label{thm:logSob}
  Let a positive random variable $G\in L\log L(E_A)$. Then,
  \begin{equation}\label{log-sob}
    \esp{G\log G}-\esp{G}\log \esp{G}\le \sum_{k\in A}\esp{\frac{|D_kG|^2}{\esp{G\,|\,\g{k}}}}.
  \end{equation}
\end{theorem}

In the usual vector calculus on $\R^{3}$, the Helhmoltz decomposition stands
that a sufficiently smooth vector field can be resolved in the sum of a curl-free
vector field and a divergence-free vector field. We have here the exact
counterpart with our definition of gradient.
\begin{theorem}[Helhmoltz decomposition]
  \label{thm:helmholtz}
  Let $U\in \DD(l^2(A))$. There exists a unique couple $(\varphi, V)$ where
  $\varphi\in L^2(E_A)$ and $V\in L^2(\AtEA)$ such that $\esp{\varphi}=0$,
  $\delta V=0$ and
  \begin{equation*}\label{helmholtz}
    U_a=D_a\varphi + V_a, 
  \end{equation*}
  for any $a\in A$.
\end{theorem}

\section{Dirichlet structures}
\label{sec:dirichlet-structures}
We now show that the usual Poisson and Brownian Dirichlet structures, associated
to their respective gradient, can be retrieved as limiting structures of
convenient approximations. This part is directly inspired
by~\cite{bouleau_theoreme_2005} where with our notations, the $X_{a}$'s are
supposed to be real valued, independent and identically distributed and the
gradient be the ordinary gradient on $\R^{A}$.

For the definitions and properties of Dirichlet calculus, we refer to the first
chapter of~\cite{bouleau-hirsch}. On $(E_A, \P_A)$, we have already implicitly
built a Dirichlet structure, i.e. a Markov process $X$, a semi-group $P$ and a
generator $L$ (see subsection~\ref{sec:ornst-uhlenb-semi}). It remains to define
the Dirichlet form $\mathcal E_A$ such that $\mathcal E_A(F)=\esp{F\, LF}$ for
any sufficiently regular functional $F$.
\begin{definition}
  For $F\in \DD$, define
  \begin{equation*}
    \mathcal E_A(F)=\esp{\sum_{a\in A} |D_aF|^2}=\|DF\|_{L^2(\AtEA)}^2.
  \end{equation*}
\end{definition}
The integration by parts formula means that this form is closed. Since we do not
assume any property on $E_a$ for any $a\in A$ and since we do not seem to have a
product rule formula for the gradient, we cannot assert more properties for
$\mathcal E_A$. However, following~\cite{bouleau_theoreme_2005}, we now show
that we can reconstruct the usual gradient structures on Poisson and Wiener
spaces as well chosen limits of our construction. For these two situations, we
have a Polish space $W$, equipped with $\mathcal B$ its Borelean $\sigma$-field
and a probability measure $\P$. There also exists a Dirichlet form $\mathcal E$
defined on a set of functionals $\DD$. Let $(E_N,\, \mathcal A_N)$ be a sequence
of Polish spaces, all equipped with a probability measure $\P_N$ and their own
Dirichlet form $\mathcal E_N$, defined on $\DD_N$. Consider maps $U_N$ from
$E_N$ into $W$ such that $(U_N)_*\P_N$, the pullback measure of $\P_N$ by $U_N$,
converges in distribution to $\P$. We assume that for any $F\in \DD$, the map
$F\circ U_N$ belongs to $\DD_N$. The image Dirichlet structure is defined as
follows. For any $F\in \DD$,
\begin{equation*}
  \mathcal E^{U_N}(F)=\mathcal E_N(F\circ U_N).
\end{equation*}
We adapt the following definition from~\cite{bouleau_theoreme_2005}.
\begin{definition}
  With the previous notations, we say that $((U_N)_*\P_N,\ N\ge 1)$ converges as
  a Dirichlet distribution whenever for any $F \in \Lip\cap \DD$,
  \begin{equation*}
    \lim_{N\to \infty} \mathcal E^{U_N}(F)=\mathcal E(F).
  \end{equation*}
\end{definition}

\subsection{Poisson point process}
\label{sec:poiss-point-proc}

Let $\bY$ be a compact Polish space and $\fN_\bY$ be the set of weighted
configurations, i.e. the set of locally finite, integer valued measures
on~$\bY$. Such a measure is of the form
\begin{equation*}
  \omega=\sum_{n=1}^\infty p_n \,\varepsilon_{\zeta_n},
\end{equation*}
where $(\zeta_n,\, n\ge 1)$ is a set of distinct points in $\bY$ with no
accumulation point, $(p_n,\, n\ge 1)$ any sequence of positive integers. The
topology on $\fN_\bY$ is defined by the semi-norms
\begin{equation*}
  p_f(\omega)=\left|\sum_{n=1}^\infty  p_n \, f(\zeta_n)\right|,
\end{equation*}
when $f$ runs through the set of continuous functions on $\bY$. It is known (see
for instance~\cite{kallenberg83}) that $\fN_\bY$ is then a Polish space for this
topology. For some finite measure $\bM$ on $\bY$, we put on $\fN_\bY$, the
probability measure $\P$ such that the canonical process is a Poisson point
process of control measure $\bM$, which we consider without loss of generality,
to have total mass $\bM(\bY)=1$.

On $\fN_\bY$, it is customary to consider the difference gradient (see
\cite{Decreusefond2014452,nualart88_1,MR2531026}): For any $x\in \bY$, any
$\omega\in \fN_\bY$,
\begin{equation}
\label{eq_gradient_spa_v2:18}  D_xF(\omega)=F(\omega+\varepsilon_x)-F(\omega).
\end{equation}
Set
\begin{align}
  \DD_{P}&=\left\{F\, :\,\fN_\bY\to \R \text{ such that }
           \esp{\int_\bY |D_xF|^2\dif \bM(x)}<\infty\right\},  \notag\\
  \intertext{and for any $F\in \DD_{P}$,}
  \mathcal E(F)&=\esp{\int_\bY |D_xF|^2\dif \bM(x)}.\label{eq:1}
\end{align}

To see the Poisson point process as a Dirichlet limit, the idea is to partition
the set $\bY$ into $N$ parts, $C_1^N,\cdots,C_N^N$ such that $\bM(C_k^N)=p_k^N$
and then for each $k\in \{1,\cdots,N\}$, take a point $\zeta_k^N$ into $C_k^N$
so that the Poisson point process $\omega$ on $\bY$ with intensity measure $\bM$
is approximated by
\begin{equation*}
  \omega^N=\sum_{k=1}^N \omega(C_k^N)\ \varepsilon_{\zeta_k^N}.
\end{equation*}
We denote by $\P_N$ the distribution of $\omega^N$. By computing its Laplace
transform, it is clear that $\P_N$ converges in distribution to $\P$. It remains
to see this convergence holds in the Dirichlet sense for the sequence of
Dirichlet structures induced by our approach for independent random variables.

Let $(\zeta_k^N, \, k=1,\cdots,N)$ (respectively $(p_k^N, \, k=1,\cdots,N)$) be
a triangular array of points in $\bY$ (respectively of non-negative numbers)
such that the following two properties hold: \\
1) the $p_k^N$'s tends to $0$
uniformly:
\begin{equation}\label{eq_Article-part1:1}
  p^N=\sup_{k\le N}p_k^N =O\left(\frac{1}{N}\right);
\end{equation}
2) the $\zeta_k^N$'s are sufficiently well spread so that we have convergence of
Riemann sums: For any continuous and $\bM$-integrable function $f\, :\, \bY\to
\R$, we have
\begin{equation}\label{eq_Article-part1:2}
  \sum_{k=1}^N f(\zeta_k^N)\, p_k^N\xrightarrow{N\to \infty}\int f(x)\dif\bM(x).
\end{equation}
Take $f=1$ implies that $\sum_{k} p_k^N$ tends to $1$ as $N$ goes to infinity.

\noindent
For any $N$ and any $k\in \{1,\cdots,N\}$, let $\mu_k^N$ be the Poisson
distribution on $\mathbf N$, of parameter $p_k^N$. In this situation, let
$E_N={\mathbf N}^N$ with $\mu^N=\otimes_{k=1}^N \mu_k^N$. That means we have
independent random variables $M_1^N, \cdots, M_N^N$, where $M_k^N$ follows a
Poisson distribution of parameter $p_k^N$ for any $k\in \{1,\cdots,N\}$. We turn
these independent random variables into a point process by the map $U_N$ defined
as
\begin{align*}
  U_N\, :\, {\mathbf N}^N &\longrightarrow \fN_\bY\\
  (m_1,\cdots,m_N) & \longmapsto \sum_{k=1}^N m_k\ \varepsilon_{\zeta_k^N}.
\end{align*}

\begin{lemma}
  \label{lem_Article-part1:1}
  For any $F\in \DD_{P}$,
  \begin{multline}\label{eq_Article-part1:3}
    \mathcal E^{U_N}(F)\\
    =\sum_{m=1}^N\sum_{\ell=0}^\infty\esp{\left(\sum_{\tau=0}^\infty\Bigl(F(\omega_{(m)}^N+\ell\varepsilon_{\zeta_m^N})-
        F(\omega_{(m)}^N+\tau\varepsilon_{\zeta_m^N})\Bigr)\mu_m^{N}(\tau)\right)^2}
    \mu_m^N(\ell),
  \end{multline}
  where $\omega_{(m)}^N=\sum_{k\neq m}M_k^N\varepsilon_{\zeta_k^N}$.
\end{lemma}
\begin{proof}
  According its very definition,
  \begin{equation*}
    \mathcal E^{U_N}(F)=\sum_{m=1}^{N}\esp{\left(F(\omega_{(m)}^N+M_{m}^{N}\varepsilon_{\zeta_m^N})-\sum_{\tau=0}^{\infty}F(\omega_{(m)}^N+\tau\varepsilon_{\zeta_m^N})\mu_m^{N}(\tau)\right)^{2}}.
  \end{equation*}
  The result follows by conditioning with respect to $M_{m}^{N}$, whose law is
  $\mu_{m}^{N}$.
\end{proof}

Since the vague topology on $\fN_\bY$ is metrizable, one could define Lipschitz
functions with respect to this distance. However, this turns out to be not
sufficient for the convergence to hold.
\begin{definition}
  A function $F\, :\:\,\fN_\bY\to \R$ is said to be $\TVlip$ if $F$ is
  continuous for the vague topology and if for any $\omega, \, \eta \in
  \fN_\bY$,
  \begin{equation*}
    |F(\omega)-F(\eta)|\le \operatorname{dist}_{\text{TV}}(\omega,\, \eta),
  \end{equation*}
  where $\operatorname{dist}_{\text{TV}}$ represents the distance in total
  variation between two point measures, i.e. the number of distinct points
  counted with multiplicity.
\end{definition}
\begin{theorem}
  \label{thm_Article-part1:1}
  For any $F\in \TVlip\, \cap\, \DD_{P}$, with the notations of
  Lemma~[\ref{lem_Article-part1:1}] and~\eqref{eq:1},
  \begin{equation*}
    \mathcal E^{U_N}(F)\xrightarrow{N\to \infty} \mathcal E(F).
  \end{equation*}
\end{theorem}

\subsection{Brownian motion}
\label{sec:brownian-motion}
For details on Gaussian Malliavin calculus, we refer
to~\cite{nualart.book,ustunel2000}. We now consider $\P$ as the Wiener measure
on $W=\mathcal C_0([0,1];\R)$. Let $(h_k,\, k\ge 1)$ be an orthonormal basis of
the Cameron-Martin space $H$,
\begin{equation*}
  H=\left\{f\, :[0,1]\to \R, \ \exists \dot f \in L^2 \text{ with }
  f(t)=\int_0^t \dot f(s)\dif s\right\} \text{ and } \|f\|_H=\|\dot f\|_{L^2}.
\end{equation*}
A function $F\, :\, W\to \R$ is said to be cylindrical if it is of the form
\begin{equation*}
  F(\omega)=f(\delta_B v_1,\cdots, \delta_B v_n),
\end{equation*}
where $v_1,\cdots,v_n$ belong to $H$,
\begin{equation*}
 \delta_B v =\int_{0}^{1}v(s)\dif \omega(s)
\end{equation*}
is the Wiener integral of $v$
and $f$ belongs to the Schwartz space $\mathcal S(\R^n)$. For $h\in H$,
\begin{equation}
  \label{eq_gradient_spa_v2:19}
  \nabla_h F(\omega)=\sum_{k=1}^n \frac{\partial f}{\partial x_k}(\delta_B
  v_1,\cdots, \delta_B v_n)\, h_k.
\end{equation}
The map $\nabla$ is closable from $L^2(W;\R)$ to $L^2(W;H)$. Thus, it is
meaningful to define $\DD_{B}$ as the closure of cylindrical functions for the
norm
\begin{equation*}
  \|F\|_{1,2}=\|F\|_{L^2(W)}+\|\nabla F\|_{L^2(W;H)}.
\end{equation*}
\begin{definition}
  A function $F\, :\, W\to \R$ is said to be $\HC$ if
  \begin{itemize}
  \item for almost all $\omega \in W$, $h\longmapsto F(\omega+h)$ is a
    continuous function on $H$,
  \item for almost all $\omega\in W$, $h\longmapsto F(\omega+h)$ is continuously
    Fr\'echet differentiable and this Fr\'echet derivative is continuous from
    $H$ into $\R\otimes H$.
  \end{itemize}
  We still denote by $\nabla F$ the element of $H$ such that
  \begin{equation*}
    \left.\frac{d}{d\tau} F(\omega+\tau h)\right|_{\tau=0}=\langle \nabla
    F(\omega),\, h\rangle_H.
  \end{equation*}
\end{definition}
\noindent
For $N\ge 1$, let
\begin{equation*}
  e_k^N(t)=\sqrt{N}\ \car_{[(k-1)/N,\, k/N)}(t) \text{ and }
  h_k^N(t)=\int_0^t e_k^N(s)\dif s.
\end{equation*}
The family $(h_k^N,\, k=1,\cdots,N)$ is then orthonormal in $H$. For $(M_k,\, k=1,\cdots,N)$
a sequence of independent identically distributed random variables, centered
with unit variance, the random walk
\begin{equation*}
  \omega^N(t)=\sum_{k=1}^N M_k\,h_k^N(t),\text{ for all }t\in [0,1],
\end{equation*}
is known to converge in distribution in $W$ to $\P$. Let $E_N=\R^N$ equipped
with the product measure $\P_N=\otimes_{k=1}^N\nu$ where $\nu$ is the standard
Gaussian measure on $\R$. We define the map $U_N$ as follows:
\begin{align*}
  U_N\, :\, E_N& \longrightarrow W\\
  m=(m_1,\cdots,m_N)& \longmapsto \sum_{k=1}^N m_k \, h_k^N. 
\end{align*}
It follows from our definition that:
\begin{lemma}
  \label{lem_Article-part1:dirichlet_rw}
  For any $F\in L^2(W;\R)$,
  \begin{equation*}
    \mathcal E^{U_N}(F)=\sum_{k=1}^N \esp{\Bigl(F(\omega^N)-\mathbf E'\left[F(\omega^N_{(k)}+M'_k\,h_k^N)\right]\Bigr)^2},
  \end{equation*}
  where
  \begin{math}
    \omega^N_{(k)}=\omega^N-M_k\, h_k^N
  \end{math}
  and $M'_k$ is an independent copy of $M_k$. The expectation is taken on the
  product space $\R^{N+1}$ equipped with the measure $\P_N\otimes \nu$.
\end{lemma}
The definition of Lipschitz function we use here is the following:
\begin{definition}
  A function $F\, :\, W\to \R$ is said to be Lipschitz if it is $\HC$ and for
  almost all $\omega\in W$,
  \begin{equation*}
    |\langle \nabla F(\omega),\, h\rangle|\le \|\dot h \|_{L^1}.
  \end{equation*}
\end{definition}
In particular since $e_k^N\ge 0$, this implies that
\begin{equation*}
  |\langle \nabla F(\omega),\, h_k^N\rangle| \le h_k^N(1)-h_k^N(0)=\frac{1}{\sqrt{N}}\cdotp
\end{equation*}
For $F\in \DD_{B}\cap \HC$, we have
\begin{equation}
  \label{eq_Article-part1:5}
  F(\omega+ h)-F(\omega)= \langle \nabla F(\omega),\, h\rangle_H
  +\|\dot h\|_{L^1}\,\varepsilon(\omega,h),
\end{equation}
where $\varepsilon(\omega,h)$ is bounded and goes to $0$ in $L^2$, uniformly
with as $\|\dot h \|_{L^1}$ tends to $0$.
\begin{theorem}
  \label{thm:donsker}
  For any $F\in \DD_{B}\cap \HC$,
  \begin{equation*}
    \mathcal E^{U_N}(F)\xrightarrow{N\to \infty} \esp{\|\nabla
      F\|_H^2}=\mathcal E(F). 
  \end{equation*}
\end{theorem}

\section{Applications}
\label{sec:appl-perm}

\subsection{Representations}
We now show that our Clark decomposition yields interesting decomposition of
random variables. For $U$-statistics, it boils down to the Hoeffding
decomposition. 

\label{sec:representations}
\begin{definition} For an integer $m$, let $h:\mathbb{R}^m\rightarrow\mathbb{R}$
  be a symmetric function, and $X_1,\cdots,X_n$, $n$ random variables supposed
  to be independent and identically distributed. The $U$-statistics of degree
  $m$ and kernel $h$ is defined, for any $n\ge m$ by
  \begin{equation*}
    U_n=U(X_1,\cdots,X_n)=\binom{n}{m}^{-1}\sum_{A\in([n],m)}h(X_A)
  \end{equation*}
  where $([n],m)$ denotes the set of ordered subsets $A\subset
  [n]=\{1,\cdots,n\}$, of cardinality~$m$.

  More generally, for a set $B$, $(B,m)$ denotes the set of subsets of $B$ with
  $m$ elements.
\end{definition}
If $\esp{|h(X_1,\cdots,X_m)|}$ is finite, we define $h_m=h$ and for $1\le k\le
m-1$,
\begin{equation*}
  h_k(X_1,\cdots,X_k)=\esp{h(X_1,\cdots,X_m)\, |\, X_1,\cdots,X_k}.
\end{equation*}
Let $\theta=\esp{h(X_1,\cdots,X_m)}$, consider $ g_1(X_1)=h_1(X_1)-\theta,$ and
\begin{equation*}
  g_k(X_1,\cdots,X_k)=h_k(X_1,\cdots,X_k)-\theta-\displaystyle{\sum_{j=1}^{k-1}\sum_{B\in ([k],j)}}g_j(X_B),
\end{equation*}
for any $1\pp k\pp m$.
%
Since the variables $X_1,\cdots,X_n$ are independent and identically
distributed, and the function $h$ is symmetric, the equality
\begin{equation*}\label{eq_gradient:7}
  \esp{h(X_{A\cup B})\, |\, X_B}=\esp{h(X_{C\cup B})\, |\, X_B},
\end{equation*}
holds for any subsets $A$ and $C$ of $[n]\backslash B$, of cardinality $n-k$.
\begin{theorem}[Hoeffding decomposition of U-statistics,
  \protect\cite{MR1472486}]
  \label{thm_gradient:1}
  For any integer $n$, we have
  \begin{equation}
    U_n=\theta+\sum_{k=1}^{m}H^{(k)}_n
  \end{equation}
  where $H^{(k)}_n$ is the $U$-statistics based on kernel $g_k$, i.e. defined by
  \begin{equation*}
    H^{(k)}_n=\binom{n}{k}^{-1}\sum_{B\subset([n],k)}g_k(X_B).
  \end{equation*}
\end{theorem}

%


As mentioned above, reversing the natural order of $A$, provided that it exists,
can be very fruitful. We illustrate this
idea by the decomposition of the number of fixed points of a random permutation under
Ewens distribution. It could be applied to more complex functionals of
permutations but to the price of increasingly complex computations.

For every integer $N$, denote by $\mathfrak{S}_N$ the space of permutations on
$\{1,\cdots,N\}$. We always identify $\mathfrak{S}_N $ as the subgroup of
$\mathfrak{S}_{N+1}$ stabilizing the element $N+1$. For every
$k\in\{1,\cdots,N\}$, define $\I_k=\{1,\cdots,k\}$ and
\begin{equation*}
  \I=\I_1\times \I_2\times \cdots \times \I_N .
\end{equation*}
The coordinate map from $\I$ to $\I_k$ is denoted by $I_k$. Following
~\cite{Kerov2004a}, we have
\begin{theorem}
  There exists a natural bijection $\Gamma$ between $\I$ and $\mathfrak{S}_N$.
\end{theorem}
\begin{proof}
  To a sequence $(i_1,\cdots,i_N)$ where $i_k\in \I_k$, we associate the
  permutation
  \begin{equation*}
    \Gamma(i_1,\cdots,i_N)=  (N,\, i_N)\circ(N-1,\,i_{N-1})\ldots
    \circ (2,i_2).
  \end{equation*}
  where $(i,j)$ denotes the transposition between the two elements $i$ and $j$.

  To an element $\sigma_N\in \mathfrak S_N$, we associate $i_N=\sigma_N(N)$.
  Then, $N$ is a fixed point of $\sigma_{N-1}=(N,\, i_N)\circ \sigma_N$, hence
  it can be identified as an element $\sigma_{N-1}$ of $\mathfrak S_{N-1}$.
  Then, $i_{N-1}=\sigma_{N-1}(N-1)$ and so on for decreasing indices.

  It is then clear that $\Gamma$ is one-to-one and onto.
\end{proof}
In ~\cite{Kerov2004a}, $\Gamma$ is described by the following rule: Start with
permutation $\sigma_{1}=(1),$ if at the $N$-th step of the algorithm, we have
$i_{N}=N$ then the current permutation is extended by leaving $N$ fixed,
otherwise, $N$ is inserted in $\sigma_{{N-1}}$ just before $i_{N}$ in the cycle
of this element. This construction is reminiscent of the Chinese restaurant
process (see~\cite{MR1177897}) where $i_{N}$ is placed immediately after ${N}$.
An alternative construction of permutations is known as the Feller coupling (see
\cite{MR1177897}). In our notations, it is given by
\begin{equation*}
  \sigma_{1}=(1); \
  \sigma_{{N}}=\sigma_{N-1}\circ (\sigma_{N-1}^{{-1}}(i_{N}),\ N). 
\end{equation*}
\begin{definition}[Ewens distribution]
  For some $t\in \R^+$, for any $k\in\{1,\cdots,N\}$, consider the measure $\P_k$ defined on
  $\I_k$ by
  \begin{equation*}
    \P_k(\{j\})=
    \begin{cases}
      \dfrac{1}{t+k-1} & \text{ if } j\neq k,\\
      &\\
      \dfrac{t}{t+k-1} & \text{ for } j=k.
    \end{cases}
  \end{equation*}
  Under the distribution $\P=\otimes_{k}\P_k$, the random variables $(I_k, \,
  k=1,\cdots,N)$ are independent with law given by
  \begin{math}
    \P(I_k=j)=\P_k(\{j\}),
  \end{math}
  for any $k$.

  The Ewens distribution of parameter $t$ on $\mathfrak S_N$, denoted by $\P^t$,
  is the push-forward of $\P$ by the map $\Gamma$.
 
\end{definition}
A moment of thought shows that a new cycle begins in the first construction for
each index where $i_k=k$. Moreover, it can be shown that
\begin{theorem}[see \protect\cite{Kerov2004a}]\label{thewens}
  For any $\sigma \in \mathfrak S_N$,
  \begin{equation*}\label{ewens}
    \P^t(\{\sigma\})=\frac{t^{\textrm{cyc}(\sigma)}}{(t+1)(t+2)\times\cdots\times(t+N-1)},
  \end{equation*}
  where $\text{cyc}(\sigma)$ is the number of cycles of $\sigma$.
\end{theorem}
For any $F$, a measurable function on $\mathfrak S_N$, we have the following
diagram
\begin{center}
  \begin{tikzpicture}
    \matrix (m) [matrix of math nodes,row sep=3em,column sep=4em,minimum
    width=2em] {
      (\I,\, \otimes_{k=1}^N \P_k)& \\
      (\mathfrak S_N,\, \P^t) & \R \\}; \path[-stealth,shorten <=2pt] (m-1-1)
    edge node[left] {$\Gamma$} (m-2-1) edge node [right] {$\quad \tilde F=F\circ
      \Gamma$} (m-2-2); \path[-stealth] (m-2-1) edge node [below] {$F$} (m-2-2);
  \end{tikzpicture}
\end{center}
We denote by $i=(i_1,\cdots, i_N)$ a generic element of $\I$ and by
$\sigma=\Gamma(i)$.

Let $C_1(\sigma)$ denote the number of fixed points of the permutation $\sigma$
and $\tilde C_1=C_1\circ \Gamma$. For any $k\in\I_N$, the random variable
$U_k(\sigma)$ is the indicator of the event ($k$ is a fixed point of $\sigma$)
and let $\tilde U^N_k=U_k \circ \Gamma$.
%
%
The Clark formula with reverse filtration shows that we can write $\tilde
U^N_{k}$ as a sum of centered orthogonal random variables as in the Hoeffding
decomposition of U-statistics (see Theorem~\ref{thm_gradient:1}).
\begin{theorem}\label{U_k}
  For any $k\in\{1,\cdots,N\}$,
  \begin{equation}\label{eq:def_de_uk}
    \tilde  U_k=\car_{(I_k=k)}\car_{(I_m\neq k,\ m\in\{k+1,\cdots, N\})}.
  \end{equation}
  and under $\P^t$, $\tilde U^N_k$ is Bernoulli distributed with parameter $t
  p_k\alpha_k$, where for any $k\in\{1,\cdots,N\}$,
  \begin{equation*}
    p_k=\dfrac{1}{t+k-1} \text{ and } \a_k=\prod_{j=k+1}^N\frac{j-1}{t+j-1}\cdotp
  \end{equation*}
  Moreover,
  \begin{multline*}
    \tilde U^N_{k}=t p_k\alpha_k + \Bigl(\car_{(I_k=k)}-tp_k\Bigr)\prod_{{m=k+1}}^{N}\car_{(I_m\neq k)}\\
    -tp_k\ \sum_{j=1}^{N-k-1}\ \frac{t+k-1}{t+k+j-2}\ \Bigl(\car_{(I_{k+j}=
      k)}-p_{k+j}\Bigr)\ \prod_{l=j+1}^{N-k}\car_{(I_{k+l}\neq k)}.
  \end{multline*}
\end{theorem}
\noindent
Since
\begin{equation*}
  \tilde C_1=\sum_{k=1}^N \tilde U^N_k,
\end{equation*}
we retrieve the result of~\cite{MR2032426}:
\begin{equation*}
  \esp{\tilde C_1}=\frac{tN}{t+N-1},
\end{equation*}
and the following decomposition of $\tilde C_{1}$ can be easily deduced from the
previous theorem.
\begin{theorem}
  \label{thm:decompositionC1}
  We can write
  \begin{multline*}
    \begin{aligned}
      \tilde C_1&=t\left( 1-\frac{t-1}{N+t-1} \right) +\sum_{l=1}^ND_l\tilde U^N_l +\sum_{l=2}^{N} \frac{t}{t+l-2} \ D_l\left( \sum_{k=1}^{l-1}\prod_{m=l}^N\car_{(I_m\neq k)} \right)\\
      &=t\left( 1-\frac{t-1}{N+t-1} \right) +\sum_{l=1}^N
      (\car_{(I_l=l)}-\frac{t}{t+l-1})\prod_{m=l+1}^N\car_{(I_m\neq l)}
    \end{aligned}\\
    -\sum_{l=2}^{N-1} \frac{t}{t+l-2}\sum_{k=1}^{l-1} \left(\car_{(I_l=
        k)}-\frac{1}{t+l-1} \right)\prod_{m=l+1}^N\car_{(I_m\neq k)}.
  \end{multline*}

\end{theorem}
\begin{remark}
   Note that such a decomposition with the
natural order on $\N$ would be infeasible since the basic blocks of the definition of
$\tilde{C}_{1}$, namely the $\tilde{U}_{k}$, are anticipative (following the
vocabulary of Gaussian Malliavin calculus), i.e.
$\tilde{U}_{k}\in \sigma(I_{k+l},l=0,\cdots,N-k)$.
\end{remark}
This decomposition can be used to compute the variance of $\tilde C_1$. To the
best of our knowledge, this is the first explicit, i.e. not asymptotic,
expression of it.
\begin{theorem}
  \label{thm:varianceC1}
  For any $t\in \R$, we get
  \begin{equation*}
    \var[\tilde C_1]=\frac{Nt}{t+N-1}\left(\frac{t}{t+N-1}+1-
      \frac{2t^2}{N}\ \sum_{k=1}^{N}\frac{1}{t+k-1}\right)
    \cdotp
  \end{equation*}
\end{theorem}
We retrieve
\begin{equation*}
  \var{[\tilde C_1]}\xrightarrow[N\rightarrow\infty]{}t,
\end{equation*}
as can be expected from the Poisson limit.
\subsection{Stein-Malliavin criterion }
\label{sec:stein}
For $(E,d)$ a Polish space, let $\Prob_{1}(E)$ the set of probability measures
on $E$. It is usually equipped with the weak convergence generated by the
semi-norms
\begin{equation*}
  p_{f}(\P)=\left|\int_{E}f\dif \P\right|
\end{equation*}
for any $f$ bounded and continuous from $E$ to $\R$. Since $E$ is Polish, we can find a denumerable
family of bounded continuous functions $(f_{n},\, n\ge 1)$ which generates the
Borelean $\sigma$-field on $E$ and the topology of the weak convergence can be
made metric by considering the distance:
\begin{equation*}
  \rho(\P,\Q)=\sum_{n=1}^{\infty} 2^{-n}\,\psi( p_{f_{n}}(\P-\Q))
\end{equation*}
where $\psi(x)=x/(1+x)$.
Unfortunately, this
definition is not prone to calculations so that it is preferable to use the
Kolmogorov-Rubinstein (or Wasserstein-1) distance defined by
\begin{equation*}
  \kappa (\P,\Q)=\sup_{\varphi\in \Lip_{1}}\left| \int_{E}\varphi\dif \P-\int_{E}\varphi\dif \Q \right|
\end{equation*}
where
\begin{equation*}
  \varphi\in\Lip_{r}\Longleftrightarrow \sup_{x\neq y\in E}\frac{|\varphi(x)-\varphi(y)|}{d(x,y)}\le r.
\end{equation*}
Theorem 11.3.1 of \cite{MR982264} states that the distances $\kappa$ and $\rho$
yield the same topology. When $E=\R$, the Stein's method is one efficient way to
compute the $\kappa$ distance between a measure and the Gaussian distribution.
If $E=\R^{n}$, for technical reasons, it is often assumed that the test
functions are more regular than simply Lipschitz continuous and we are led to
compute
\begin{equation*}
  \kappa_{\FF}(\P,\Q)=\sup_{\varphi\in \FF}\left| \int_{E}\varphi\dif \P-\int_{E}\varphi\dif \Q \right|
\end{equation*}
where $\FF$ is a space included in $\Lip_{1}$ like the set of $k$-times
differentiable functions with derivatives up to order $k$ bounded by $1$.

The setting in which we need to compute a KR distance is very often the
situation in which we have another Polish space $G$ with a probability measure
$\mu$ and a random variable $F$ with value in $E$. The objective is then to
compare some measure $\P$ on $E$ and $\P_{F}=F_{*}\mu$ the distribution of $F$, i.e. the
push-forward of $\mu$ by the application $F$. This means that we have to compute
\begin{equation}\label{eq_gradient_spa_v2:3}
  \sup_{\varphi\in \FF}\left| \int_{E}\varphi\dif \P-\int_{G}\varphi\circ F\dif \mu \right|.
\end{equation}
As mentioned in Section~\ref{sec:introduction}, when using the Stein's method,
we first characterize $\P$ by a functional identity and then use different
tricks to transform \eqref{eq_gradient_spa_v2:3} in a more tractable expression.
The usual tools are exchangeable pairs, coupling or Malliavin integration by
parts. For the latter to be possible requires that we do have a Malliavin
structure on the measured space $(G,\mu)$. In
\cite{MR2520122,taqqu}, generic theorems are given which link $
\kappa_{\FF}(\P,\P_{F})$ with some functionals of the gradient of $F$. For
instance, if $(G,\mu)$ is the space of locally finite configurations on a space
$\mathfrak g$, equipped with the Poisson distribution of control measure $\sigma$
and $\P$ is the Gaussian distribution in $\R$,
\begin{multline}\label{eq_gradient_spa_v2:4}
\kappa_{\FF}(\P,\P_{F})\pp \esp{\left|1-\int_{\mathfrak{g}}D_zF\,D_zL^{-1}F\dif\sigma(z)\right|}\\
+\int_{\mathfrak{g}}\esp{|D_zF|^2|D_zL^{-1}F|}\dif\sigma(z),
\end{multline}
where $D$ is the Poisson-Malliavin gradient (see Eqn.~\eqref{eq_gradient_spa_v2:18}), $L=D^{*}D$ the associated generator and the Stein class $\mathcal F$ is the space of twice differentiable functions with first derivative
 bounded by $1$ and second order derivative bounded by $2$. In
\cite{Doebler2016a}, an analog result is given when $\P$ is a Gamma distribution
and $(G,\mu)$ is either a Poisson or a Gaussian space. To the best of our
knowledge, when $\mu$ is the distribution of a family of independent random
variables, the distance $\kappa_{\FF}(\P,\P_{F})$ is evaluated through
exchangeable pairs or coupling, which means to construct an ad-hoc structure for
each situation at hand. We intend to give here an exact analog to
\eqref{eq_gradient_spa_v2:4} in this situation using only our newly defined operator~$D$. Our first result concerns the
Gaussian approximation. To the best of our knowledge, there does not yet exist a
Stein criterion for Gaussian approximation which does not rely on exchangeable
pairs or any other sort of coupling.

\begin{remark}
  In what follows, we deal with functions $F$ defined on $E_{A}$, that means
  that $F$ is a function of $X_{A}$ and as such, we should use the notation
  $F(X_{A})$. For the sake of notations, we identify $F$ and $F(X_{A})$. 
\end{remark}
\begin{theorem}
  \label{thm:3.1Gaussian}Let $\P$ denote the standard Gaussian distribution on
  $\R$. For any $F\, :\, E_{A}\to\R$ such that $\esp{F}=0$ and $F\in \dom D$. Then,
  \begin{multline*}
    \kappa_{{\FF}}(\P,\P_{F})\le \esp{\left|1-\sum_{a\in A} D_{a}F \ (-D_{a}L^{-1})F      \right|}\\
    +  \sum_{a\in A}\esp{\int_{E_{A}}\Bigl( F-F(X_{A\neg a};x)\Bigr)^{2}
      \dif \P_{a}(x) \ |D_{a}L^{-1}F|}.
  \end{multline*}
\end{theorem}
The proof of this version follows exactly the lines of the proof of Theorem~3.1
in \cite{MR2520122,taqqu} but we can do slightly better by changing a detail in
the Taylor expansion.
\begin{theorem}
   \label{thm:3.1Gaussianbis}Let $\P$ denote the standard Gaussian distribution on
  $\R$. For any $F\, :\, E_{A}\to\R$ such that $\esp{F}=0$ and $F\in \dom D$. Then,
  \begin{multline}
    \label{eq_gradient_spa_v2:17}
    \kappa_{{\FF}}(\P,\P_{F})
    \le \sup_{\psi\in\Lip_{2}}\esp{\psi(F)-\sum_{a\in A}\psi(F(X'_{\neg a}))
      D_{a}F (-D_{a}L^{-1})F}
    \\+  \sum_{a\in A}\esp{\int_{E_{A}}\Bigl( F-F(X_{A\neg a};x)\Bigr)^{2}
      \dif \P_{a}(x) \ |D_{a}L^{-1}F|},
  \end{multline}
  where  $X'_{\neg a}=X_{A\neg a}\cup \{X'_{a}\}$.
\end{theorem}

This formulation may seem cumbersome, but it easily gives a close to the  usual bound in the
Lyapounov central limit theorem, with a non optimal constant (see \cite{Goldstein2010}).
\begin{corollary}
  \label{thm:lyapounov}
    Let $(X_{n},\, n\ge 1)$ be a sequence of thrice integrable, independent random
  variables. Denote
  \begin{equation*}
    \sigma_{n}^{2}=\var(X_{n}), s_{n}^{2}=\sum_{{j=1}}^{n
    }\sigma_{j}^{2} \text{ and } Y_{n}=\frac{1}{s_{n}}\sum_{{j=1}}^{n} \left(X_{j}-\esp{X_{j}}\right). 
  \end{equation*}
Then,
\begin{equation*}
  \kappa_{\FF}(\P, \P_{Y_{n}})\le   \frac{2(\sqrt{2}+1)}{s_{n}^{3}}\sum_{{j=1}}^{n    }\esp{|X_{j}-\esp{X_{j}}|^{3}}.
\end{equation*}
\end{corollary}
\begin{remark}
  If we use Theorem~\ref{thm:3.1Gaussian}, we get
  \begin{equation*}
      \kappa_{\FF}(\P, \P_{Y_{n}})\le \esp{\left| 1-\sum_{j=1}^{n}\frac{X_{j}^{2}}{s_{n}^{2}} \right|}+  \frac{2}{s_{n}^{3}}\sum_{{j=1}}^{n    }\esp{|X_{j}-\esp{X_{j}}|^{3}}
    \end{equation*}
    and the quadratic term is easily bounded only if the $X_{i}$'s are such that
    $\esp{X_{i}^{4}}$ is finite, which in view of Corollary~\ref{thm:lyapounov} is
    a too stringent condition.
\end{remark}
The functional which appears in the central limit theorem is the basic example
of U-statistics or homogeneous sums. If we want to go further and address the problem of convergence
of more general U-statistics (or homogeneous sums), we need to develop a similar apparatus for the
Gamma distribution. Recall that the Gamma distribution of parameters $r$ and $\lambda$ has density
\begin{equation*}
  f_{r,\lm}(x)=\frac{\lambda^{r}}{\Gamma(r)}\, x^{r-1}e^{-\lm x}\ \car_{\R^{+}}(x).
\end{equation*}
Let $Y_{r,\lm}\sim \Gamma(r,\lm)$, it has  mean $r/\lm$ and variance $r/\lm^2$.
Denote by $\overline{Y}_{r,\lm}=Y_{r,\lm}-r/\lm$.
As described in  \cite{Graczyk_2005},  $Z\sim
\overline{Y}_{r,\lm}=Y_{r,\lm}-r/\lm$ if and only if $\esp{L_{r,\lm}f(Z)}=0$ for any
$f$ once differentiable, where
\begin{equation*}
 L_{r,\lm}f(y)=\frac{1}{\lambda}\left( y+\frac{r}{\lm} \right)f'(y)-yf(y).
\end{equation*}
The Stein equation
\begin{equation}\label{eq_gradient_spa_v2:15} 
   L_{r,\lm}f(y) =g(y)-\esp{g(\overline{Y}_{r,\lm})}
\end{equation}
has a solution $f_{g}$ which satisfies
\begin{multline}\label{eq_gradient_spa_v2:16}
  \|f_g\|_{\infty}\pp \|g'\|_{\infty}, \ \|f'_g\|_{\infty}\pp 2\lambda\max\left(1,\frac{1}{r}\right) \|g'\|_{\infty}\\
  \text{ and }  \|f''_g\|_{\infty}\pp 2\lambda\left(\max\left(\lambda,\frac{\lambda}{r}\right)\|g'\|_{\infty}+\|g''\|_{\infty}\right),
\end{multline}
noting that $f_g$ is solution of \eqref{eq_gradient_spa_v2:15} if and only if $h_g:x\mapsto\dfrac{1}{\lambda}f\Big(x-\dfrac{r}{\lambda}\Big)$ solves
\begin{equation*}
xh'(x)+(r-\lambda x)h(x)=g(x)-\esp{g(Y_{r,\lambda})},
\end{equation*}
studied in \cite{Arras2015a,Doebler2016a}.
\begin{theorem}
  \label{thm:3.1Gamma}
Let $\mathcal F$ is the set of twice differentiable functions with first and second
derivative bounded by $1$. There exists $c>0$ such that for any  $F\in \dom D$ with $\esp{F}=0$, 
  \begin{multline}\label{eq_gradient_spa_v2:7}
    \kappa_{\FF}(\P_{F},\,\P_{\overline{Y}_{r,\lm}})\le c\,
    \esp{\left|\frac{1}{\lm}F + \frac{r}{\lm^{2}}-\sum_{a\in A} D_{a}F(-D_{a}L^{-1})F\right|}\\
    +c\,\sum_{a\in
      A}\esp{\int_{E_{A}}\Bigl( F(X_{A})-F(X_{A\neg a};x)\Bigr)^{2}
      \dif \P_{a}(x)\ |D_aL^{-1}F|}.
  \end{multline}
\end{theorem}
This theorem reads exactly as \cite[Theorem 1.5]{Doebler2016a} for Poisson
functionals and is proved in a similar fashion.
\begin{remark}
  The generalization of this result to multivariate Gamma distribution will be
  considered in a forthcoming paper. The difficulty lies in the regularity estimates of the
  solution of the Stein equation associated to multivariate Gamma distribution, which require lengthy calculations.
\end{remark}

An homogeneous sum of order $d$ is a functional of independent identically
distributed random variables $(X_{1},\cdots,X_{N_{n}})$, of the form
\begin{equation*}
  F_{n}(X_{1},\cdots,X_{N_{n}})=\sum_{1\le i_{1},\cdots,i_{d}\le N_{n}} f_{n}(i_{1},\cdots,i_{d}) \, X_{i_{1}}\ldots X_{i_{d}}
\end{equation*}
where $(N_{n},n\ge 1)$ is a sequence of integers which tends to infinity as $n$ does and  the functions $f_{n}$ are symmetric on $\{1,\cdots,N_{n}\}^{d}$ and vanish on the diagonal. The asymptotics of
these sums have been widely investigated and depend on the properties of the
function $f_{n}$. For $d=2$, see for instance \cite{Goetze2002}. In
\cite{Nourdin2010a}, the case of any value of $d$ is investigated through the
prism of universality: roughly speaking (see Theorem 4.1), if $F_{n}(G_{1},\cdots,G_{N_{n}})$ converges in
distribution when $G_{1},\cdots,G_{N_{n}}$ are standard Gaussian random variables
then $F_{n}(X_{1},\cdots,X_{N_{n}})$ converges to the same limit whenever the $X_{i}$'s are centered with
unit variance and finite third order moment and such that 
\begin{equation*}
  \max_{i}\sum_{1\le i_{2},\cdots,i_{d}\le N_{n}}f_{n}^{2}(i,i_{2},\cdots,i_{d})\xrightarrow{n\to \infty}0.
\end{equation*}
For Gaussian random variables, the functional $F_{n}$ belongs to the $d$-th
Wiener chaos. Combining the  algebraic rules of multiplication of iterated Gaussian
integrals and the Stein-Malliavin method, it is proved in \cite{Nourdin2009}
that  $F_{n}(G_{1},\cdots,G_{N_{n}})$ converges in distribution to a chi-square
distribution of parameter $\nu$ if and only if
\begin{equation*}
  \esp{F_{n}^{2}}\xrightarrow{n\to \infty} 2\nu \text{ and }  \esp{F_{n}^{4}}-12\,\esp{F_{n}^{3}}-12\nu^{2}+48\nu \xrightarrow{n\to \infty} 0.
\end{equation*}
We obtain here a related result for $d=2$ (for the sake of simplicity though the
method is applicable for any value of $d$) and a general distribution without
resorting to universality.

Let $A=\{1,\cdots,n\}$. For $f,g\,:\, A^2\rightarrow \R$, symmetric functions vanishing on the
diagonal,  define 
the two contractions  by
\begin{align*}
  (f\star_1^1g)(i,j)&=\sum_{k\in A}f(i,k)g(j,k),\\
    (f\star_2^1g)(i)&=\sum_{j\in A}f(i,j)g(i,j).
\end{align*}

\begin{theorem}\label{gamma}
Let $X_A=\{X_i, \, 1\pp  i\pp n\}$ be a collection of centered independent
random variables with unit variance and finite moment of order~4. Define
\begin{equation*}
  F(X_A)=\sum_{(i,j)\in A^{\neq}}f(i,j)\,X_iX_j 
\end{equation*}
where $(i,j)\in A^{\neq}$ means that we enumerate all the couples $(i,j)$ in $A^2$
with distinct components and $f$ is  a symmetric function which vanishes on the diagonal.
Let $\nu=\sum_{(i,j)}f^{2}(i,j)$. Then, there exists $c_{\nu}>0$ such that 
\begin{multline}
  \label{eq_gradient_spa_v3:1}
  \kappa_{\mathcal H}^{2}(\P_F,\,\P_{\bar Y_{\nu/2,1/2}})\le
  c_{\nu}\esp{X_{1}^{4}}^{2}\\ \times \left[\sum_{(i,a)\in A^{2}}f^{4}(i,a)+
    \|f\star_{2}^{1}f\|^{2}_{L^{2}(A)} +
    \|f-f\star_{1}^{1}f\|_{L^{2}(A^{2})}^{2}\right].
\end{multline}
\end{theorem}

We now introduce $\mathrm{Inf}_a(f)$, called the influence of the variable $a$, by
\begin{equation*}
\mathrm{Inf}_a(f)=\sum_{i\in A}f^2(i,a).
\end{equation*}
Remark that 
\begin{align*}
\sum_{i\in A}f^{4}(i,a)
&\pp \sum_{a\in A}\sum_{i}f^{2}(i,a)\,\sum_{j} f^{2}(j,a)\\
&=\sum_{a\in A}\sum_{i}f^{2}(i,a)\, \mathrm{Inf}_a(f)\\
&\pp\nu\ \underset{a\in A}\max\,\mathrm{Inf}_a(f).
\end{align*}
The same kind of computations can be made for
$\|f\star_{2}^{1}f\|_{L^{2}(A)}^{2}$. As a consequence, we get the following corollary.
\begin{corollary}
  With the same notations as above,
  \begin{equation*}
  \kappa_{\mathcal H}^{2}(\P_F,\,\P_{\bar Y_{\nu/2,1/2}})\le
  c_{\nu}\esp{X_{1}^{4}}^{2} \left[  \underset{a\in A}\max\,\mathrm{Inf}_a(f)+  \|f-f\star_{1}^{1}f\|_{L^{2}(A^{2})}^{2}\right].
\end{equation*}
\end{corollary}
The supremum of the influence is the quantity which governs the distance between
the distributions of 
$F_{n}(G_{1},\cdots, G_{N_{n}})$ and $F_{n}(X_{1},\cdots,X_{N_{n}})$ in
\cite{Nourdin2010a}, thus it  is not surprising that it still appears here.
\begin{remark}
A tedious computation shows that
\begin{multline}\label{eq_gradient_spa_v3:8}
\esp{F^4}-12\esp{F^3}-12\nu^2+48\nu\\
=\sum_{(i,j)\in A^{\neq}}\,f^{4}(i,j)\,\esp{X^4}^2\,+\,6\,\sum_{(i,j,k)\in A^{\neq}}f^{2}(i,j)f^{2}(i,k)\esp{X^4}\\
+12\,\esp{X^3}^2\,\left\{\sum_{(i,j,k)\in A^{\neq}}f^{2}(i,j)\,f(i,k)\,f(k,j)-\sum_{(i,j)\in A^{\neq}}\,f^{3}(i,j)\right\}\\
 -48\left\{\sum_{(i,j,k)\in A^{\neq} }f(i,j)f(i,k)f(k,j)-f^{2}(i,j)\right\}-12\sum_{(i,j)\in A^{\neq}}\,f^{4}(i,j).
\end{multline}
The Cauchy-Schwarz inequality entails  that the properties
\begin{equation*}
  \esp{F_{n}^4}-12\esp{F_{n}^3}-12\nu^2+48\nu \xrightarrow{n\to \infty} 0
\end{equation*}
and
\begin{equation*}
\kappa_{\mathcal H}(\P_F,\,\P_{\bar Y_{\nu/2,1/2}})   \xrightarrow{n\to \infty} 0
\end{equation*}
share the same sufficient condition:
\begin{equation*}
 \sum_{(i,a)\in A^{\neq}}f^{4}(i,a)+
    \|f\star_{2}^{1}f\|^{2}_{L^{2}(A)} +
    \|f-f\star_{1}^{1}f\|_{L^{2}(A^{2})}^{2} \xrightarrow{n\to \infty} 0.
\end{equation*}
However, we cannot go further and state a \textsl{fourth
moment theorem} as we know, that
 for Benoulli random variables, $F_{n}$ may converge to $\bar{Y}_{\nu/2,1/2}$
 while the RHS of~\eqref{eq_gradient_spa_v3:1} does not converge to $0$.
\end{remark}
As another corollary of Theorem~\ref{gamma}, we obtain the KR distance between a degenerate
U-statistics of order $2$ and a Gamma distribution. 
Compared to the more general \cite[Theorem 1.1]{Doebler2016a}, the computations
are here greatly simplified by the absence of exchangeable pairs. 
\begin{theorem}
  Let
$A=\{1,\cdots,n\}$ and $(X_{i},i\in A)$ a family of independent and  identically
distributed real-valued random variables such that
\begin{equation*}
  \esp{X_{1}}=0,\ \esp{X_{1}^{2}}=\sigma^{2}  \text{ and }  \esp{X_{1}^{4}}<\infty.
\end{equation*}
Consider the random variable
\begin{equation*}
F=\frac{2}{n-1}\sum_{(i,j)\in A^{\neq}}X_{i}X_{j}.
\end{equation*}
Then, there exists $c>0$, independent of $n$, such that 
\begin{equation}
  \label{eq_gradient_spa_v2:14}
   \kappa_{\FF}(\P_{F},\,\P_{\overline{Y}_{1/2,1/2\sigma^{2}}}) \le c\, \frac{\sigma^{2}}{\sqrt{n}}\, \esp{X_{1}^{4}}.
\end{equation}
\end{theorem}
\begin{proof}
  Take $f_{n}(i,j)=2/(n-1)$ and apply Theorem~\ref{gamma}.
\end{proof}
\begin{remark}
  The proof of Theorem~\ref{gamma} is rich of insights. In Gaussian, Poisson or
  Rademacher contexts, the computation of
  $L^{-1}F$ is easily done when there exists a chaos decomposition since  $L$ operates as
  a dilation on each chaos (see
  \cite{MR2520122,Nourdin:2012fk,
    taqqu}).
  In \cite[Lemma 3.4 and below]{Reitzner2013}, a formula for $L^{-1}$ of  Poisson driven
  U-statistics is given, not resorting to the chaos decomposition. It is based on the fact that $L$ applied to a
  U-statistics $F$ of  order $k$ yields $kF$ plus a U-statistics of order $(k-1)$.
  Then, the construction of an inverse formula can be made by induction. In our
  framework, the action of $L$ on a U-statistics yields $kF$ plus a
  U-statistics of order $k$ so that no induction seems possible. However, for
  an order $k$ U-statistics which is degenerate of order $(k-1)$, we have
  $LF=kF$. For $k=2$, this hypothesis of degeneracy is exactly the sufficient
  condition to have a convergence towards a Gamma distribution.
  
\end{remark}

\section{Proofs}
\label{sec:proofs}

\subsection{Proofs of Section \protect \ref{sec:mall-calc-indep}}
\label{sec:p2}
\begin{proof}[Proof of Theorem~\protect\ref{thm:ipp}]
  The process $\trace(DU)=(D_aU_a,\, a\in B)$ belongs to $L^2(\AtEA)$: Using the
  Jensen inequality, we have
  \begin{equation}
    \label{eq_gradient:1}
    \| \trace(DU)\|^2_{L^2(\AtEA)}=\esp{\sum_{a\in B} |D_aU_a|^2}\pp 2
    \sum_{a\in B} \esp{U_a^2}<\infty.
  \end{equation}
  Moreover,
  \begin{multline*}
    {\langle DF, U \rangle_{L^2(\AtEA)}} = \esp{\sum_{a\in A}(F-\esp{F\, |\,
        \exv_a})\ U_a} \\= \esp{\sum_{a\in B}(F-\esp{F\, |\, \exv_a})\ U_a}
    =\esp{F \ \sum_{a\in B}(U_a-\esp{U_a\, |\, \exv_a})},
  \end{multline*}
  since the conditional expectation is a projection in $L^2(E_A)$.
\end{proof}
\begin{proof}[Proof of corollary~\protect\ref{closability}]
  Let $(F_n, \, n\ge 1)$ be a sequence of random variables defined on
  $\mathcal{S}$ such that $F_n$ converges to 0 in $L^2(E_A)$ and the sequence
  $DF_n$ converges to $\eta$ in $L^2(\AtEA)$. Let $U$ be a simple process.
  From the integration by parts formula (\ref{IPP})
  \begin{align*}
    \esp{\sum_{a\in A}D_aF_n\ U_a}
    &=\esp{F_n\sum_{a\in A}D_aU_a}
  \end{align*}
  where $\displaystyle\sum_{a\in A}D_aU_a\in L^2(E_A)$ in view
  of~\eqref{eq_gradient:1}. Then,
  \begin{equation*}
    \langle \eta, U\rangle_{L^2(\AtEA)}
    =\underset{n\rightarrow \infty}\lim\esp{F_n\sum_{a\in A}D_aU_a}=0,
  \end{equation*}
  for any simple process $U$. It follows that $\eta=0$ and then the operator $D$
  is closable from $L^2(E_A)$ to $L^2(\AtEA)$.
\end{proof}

\begin{proof}[Proof of Lemma~\protect\ref{lem:boundedness}]
  Since $\sup_n\|D F_n\|_{\DD}$ is finite, there exists a subsequence which we
  still denote by $(D F_{n}, n\ge 1)$ weakly convergent in $L^2(\AtEA)$ to some
  limit denoted by $\eta$. For $k>0$, let $n_k$ be such that
  $\|F_{m}-F\|_{L^2}<1/k$ for $m\ge n_k$. The Mazur's Theorem implies that there
  exists a convex combination of elements of $(D F_{m}, m\ge n_k)$ such that
  \begin{equation*}
    \Big\|\sum_{i=1}^{M_k} \alpha^k_i DF_{m_i}-\eta\,\Big\|_{L^2(\AtEA) }<1/k.
  \end{equation*}
  Moreover, since the $\alpha^k_i$ are positive and sums to $1$,
  \begin{equation*}
    \Big\|\sum_{i=1}^{M_k} \alpha^k_i F_{m_i}-F\,\Big\|_{L^2(E_A)}\le 1/k.
  \end{equation*}
  We have thus constructed a sequence
  \begin{equation*}
    F^k=\sum_{i=1}^{M_k} \alpha^k_i F_{m_i}
  \end{equation*}
  such that $F^k$ tends to $F$ in $L^2$ and $D F^k$ converges in $L^2(\AtEA)$ to
  a limit. By the construction of $\DD$, this means that $F$ belongs to $\DD$
  and that $D F=\eta$.
\end{proof}
\begin{proof}[Proof of Theorem~\protect\ref{thm:PtADenumerable}]
  To prove the existence of $(P_{t},t\ge 0)$ for a countable set, we apply the
  Hille-Yosida theorem:
  \begin{proposition}[Hille-Yosida]
    A linear operator $L$ on $L^2(E_A)$ is the generator of a strongly
    continuous contraction semigroup on $L^2(E_A)$ if and only if
    \begin{enumerate}
    \item $\dom L$ is dense in $L^2(E_A)$.
    \item $L$ is dissipative i.e. for any $\lambda>0, F\in\dom L$,
      \begin{equation*}
        \|\lambda F- LF\|_{L^2(E_A)} \ge \lambda \|F\|_{L^2(E_A)}.
      \end{equation*}
    \item Im$(\lambda \id -L)$ dense in $L^2(E_A)$.
    \end{enumerate}
  \end{proposition}

  We know that $\cyl\subset \dom L$ and that $\cyl$ is dense in $L^2(E_A)$, then
  so does $\dom L$.

  Let $(A_n,\, n\ge 1)$ an increasing sequence of subsets of $A$ such that
  $\cup_{n\ge 1}A_n=A$. For $F\in L^2(E_A)$, let $F_n=\esp{F\, |\, \f{A_n}}$.
  Since $(F_n, \, n\ge 1)$ is a square integrable martingale, $F_n$ converges to
  $F$ both almost-surely and in $L^2(E_A)$. For any $n\ge 1$, $F_n$ depends only
  on $X_{A_n}$. Abusing the notation, we still denote by $F_n$ its restriction
  to $E_{A_n}$ so that we can consider $L_nF_n$ where $L_n$ is defined as above
  on $E_{A_n}$. Moreover, according to Lemma~\ref{lem_gradient:permutation},
  $D_aF_n=\esp{D_aF\, |\, \f{A_n}}$, hence
  \begin{multline*}
    \lambda^2\|F_n\| _{L^2(E_A)}^2\le \|\lambda F_n- L_nF_n\|_{L^2(E_{A_n})}^2=
    \esp{\left( \lambda F_n+ \sum_{a\in A}D_aF_n \right)^2}\\ =\esp{\esp{\lambda
        F+ \sum_{a\in A}D_aF\, \Bigl|\, \f{A_n}}^2} \xrightarrow{n\to
      \infty}\|\lambda F- LF\|_{L^2(E_A)}^2.
  \end{multline*}
  Therefore, point (2) is satisfied.

  Since $A_n$ is finite, there exists $G_n\in L^2(E_{A_n})$ such that
  \begin{multline*}
    F_n=(\lambda\id -L_n)G_n(X_{A_n})=\lambda G_n(X_{A_n})+\sum_{a\in
      A_n}D_aG_n(X_{A_n})\\ = \lambda \tilde G_n(X_{A})+\sum_{a\in A_n}D_a\tilde
    G_n(X_{A})= \lambda \tilde G_n(X_{A})+\sum_{a\in A}D_a\tilde G_n(X_{A}),
  \end{multline*}
  where $\tilde G_n(X_A)=G_n(X_{A_n})$ depends only on the components whose
  index belongs to $A_n$. This means that $F_n$ belongs to the range of $\lambda
  \id -L$ and we already know it converges in $L^2(E_A)$ to $F$.
\end{proof}

  \begin{proof}[Proof of Theorem~\protect\ref{thm:inversionDP_t}]
    For $A$ finite, denote by $Z_a$ the Poisson process of intensity $1$ which
    represents the time at which the $a$-th component is modified in the
    dynamics of $X$. Let $\tau_a=\inf\{t\ge 0,\ Z_a(t)\neq Z_a(0)\}$ and remark
    that $\tau_a$ is exponentially distributed with parameter $1$, hence
    \begin{multline*}
      \esp{F(X(t))\mathbf{1}_{t\ge\tau_a}\,|\, X(0)=x}\\
      \begin{aligned}
        &=(1-e^{-t})\,\esp{\int_{E_a}F(X_{\neg a}(t),x_a')\dif\P_a(x_a')\,\Big|\,X(0)=x}\\
        &=(1-e^{-t})\,\esp{\esp{F(X(t))\,|\,\exv_{a}}\,|\,X(0)=x}\\
        &=\esp{\esp{F(X(t))\,|\,\exv_{a}}\mathbf{1}_{t\ge\tau_a}\,|\,X(0)=x}.
      \end{aligned}
    \end{multline*}
%
    Then,
    \begin{multline*}
      D_aP_tF(x) =P_tF(x)-\esp{P_tF(x)\,|\,\g{a}}\\
      \begin{aligned}
        &=\esp{(F(X(t))-\esp{F(X(t))\,|\,\g{a}})\mathbf{1}_{t<\tau_a}\,|\,X(0)=x}\\
        &\qquad +\esp{(F(X(t))-\esp{F(X(t))\,|\,\g{a}})\mathbf{1}_{t\ge\tau_a}\,|\,X(0)=x}\\
        &= e^{-t}P_tD_aF(x).
      \end{aligned}
    \end{multline*}
    For $A$ infinite, let $(A_n,\, n\ge 1)$ an increasing sequence of finite
    subsets of $A$ such that $\cup_{n\ge 1}A_n=A$. For $F\in L^2(E_A)$, let
    $F_n=\esp{F\,|\, \f{A_n}}$. Since $P$ is a contraction semi-group, for any
    $t$, $P_tF_n$ tends to $P_tF$ in $L^2(E_A)$ as $n$ goes to infinity. From
    the Mehler formula, we known that $P_tF_n=P^n_tF_n$ where $P^n$ is the
    semi-group associated to $A_n$, hence
    \begin{equation}\label{eq_gradient:5}
      D_aP_tF_n=D_aP_t^nF_n=e^{-t}P_t^nD_aF_n.
    \end{equation}
    Moreover,
    \begin{align*}
      \esp{\sum_{a\in A_n}|D_aP_tF_n|^2}&=e^{-2t}\sum_{a\in A_n}\esp{|P_tD_aF_n|^2}\\
                                        &\le e^{-2t}\sum_{a\in A_n}\esp{|D_aF_n|^2}\\
                                        &= e^{-2t}\sum_{a\in A_n}\esp{|\esp{D_aF \,|\, \f{A_n}}|^2}\\
                                        &\le  e^{-2t}\sum_{a\in A_n}\esp{|D_aF|^2}\\
                                        &\le  e^{-2t}\|DF\|_{\DD}^2.
    \end{align*}
    According to Lemma~[\ref{lem:boundedness}], this means that $P_tF$ belongs
    to $\DD$. Let $n$ go to infinity in~\eqref{eq_gradient:5}
    yields~\eqref{OU-Grad}.
  \end{proof}
  \begin{proof}[Proof of Lemma~\protect\ref{Ddelta}]
    For $U$ and $V$ in $\mathcal S_0(l^2(A))$, from the integration by parts formula,
    \begin{align*}
      \esp{\delta U\ \delta V}
      &=\langle D\delta(U), V \rangle_{L^2(\AtEA)}\\
      &=\esp{\sum_{a\in A}D_a(\delta U)\,V_a}\\
      &=\esp{\sum_{(a,b)\in A^2} V_{a}\,D_{a}D_{b}U_{b}}\\
      &=\esp{\sum_{(a,b)\in A^2} V_{a}\,D_{b}D_{a}U_{b}}\\
      &=\esp{\sum_{(a,b)\in A^2}D_bV_a\, D_aU_b}
      =\esp{\trace(DU\circ DV)}.
    \end{align*}
    It follows that
    \begin{math}
      \esp{\delta U^2}\le \|U\|_{\DD(l^2(A))}^2.
    \end{math}
    Then, by density, $\DD(l^2(A))\subset \dom\delta$ and
    Eqn.~\eqref{norm_delta_1} holds for $U$ and $V$ in $\dom \delta$.
  \end{proof}
 
  \subsection{Proofs of Section \protect \ref{sec:functional}}
  \label{sec:p3}
  \begin{proof}[Proof of Lemma~\protect\ref{Lchaos1}]
    Let $k\in A$.
    Assume that $F\in\f{k}$. Then, for every $n>k$, $F$ is $\g{n}$-measurable and $D_nF=0$.\\
    Let $F\in\DD$ such that $D_nF=0$ for every $n>k$. Then $F$ is
    $\g{n}$-measurable for any $n>k$. From the equality
    $\f{k}=\underset{n>k}\cap\g{n}$, it follows that $F$ is $\f{k}$-measurable.
  \end{proof}
  \begin{proof}[Proof of Lemma~\protect\ref{chaos2}]
    For any $k\ge 1$, $\f{k}\cap \exv_k=\f{k-1}$, hence
    \begin{equation*}\label{chaos21}
      D_k\esp{F\,|\,\f{k}}=\esp{F|\f{k}}-\esp{F\,|\,\f{k-1}}=\esp{D_kF\,|\,\f{k}}.
    \end{equation*}
    The proof is thus complete.
  \end{proof}
  \begin{proof}[Proof of Theorem~\protect\ref{lchaos}]
    Let $F$ an $\mathcal{F}_n$-measurable random variable. It is clear that
    \begin{equation*}
      F-\esp{F}=\sum_{k=1}^{n}\left( \esp{F\, |\, \f{k}}-\esp{F\, |\,\f{k-1}} \right)=\sum_{k=1}^{n} D_k \esp{F\,|\,\f{k}}.
    \end{equation*}
    For $F\in \DD$, apply this identity to $F_n=\esp{F\, |\,\f{n}}$ to obtain
    \begin{equation*}
      \label{eq_g:2} F_n-\esp{F}= \sum_{k=1}^{n} D_k \esp{F\,|\,\f{k}}.
    \end{equation*}
    Remark that for $l>k$, in view of Lemma~\ref{Lchaos1},
    \begin{equation}
      \label{eq_gradient:3}
      \esp{D_k \, \esp{F\,|\,\f{k}} D_l \, \esp{F\,|\,\f{l}}}= \esp{D_lD_k \, \esp{F\,|\,\f{k}} \esp{F\,|\,\f{l}}}=0,
    \end{equation}
    since $D_k \ \esp{F\,|\,\f{k}}$ is $\f{k}$-measurable. Hence, we get
    \begin{equation*}
      \esp{|F-\esp{F}|^2}\ge     \esp{|F_n-\esp{F}|^2}=\sum_{k=1}^n \esp{D_k \esp{F\,|\,\f{k}}^2}.
    \end{equation*}
    Thus, the sequence $(D_k \esp{F\,|\,\f{k}},\, k\ge 1)$ belongs to $l^2(\N)$
    and the result follows by a limiting procedure.

    We now analyze the non-ordered situation. If $A$ is finite, each bijection
    between $A$ and $\{1,\cdots,n\}$ defines an order on $A$. Hence, there are
    $|A|\,!$ possible filtrations. Each term of the form
    \begin{equation*}
      D_{i_k}\esp{F\,|\, X_{i_1},\cdots,X_{i_k}}
    \end{equation*}
    appears $(k-1)!\, (|A|-k)!$ times since the order of $X_{i_1},\cdots,
    X_{i_{k-1}}$ is irrelevant to the conditioning. The result follows by
    summation then renormalization of the identities obtained for each
    filtration.
  \end{proof}
  \begin{proof}[Proof of Theorem~\protect\ref{lchaos:reverse}]
    Remark that
    \begin{multline*}
      D_k\,\esp{F\,|\,\h{k-1}^N}=\esp{F\,|\,\h{k-1}^N}-\esp{F\,|\,\h{k-1}^N\cap
        \exv_{k}}\\
      =\esp{F\,|\,\h{k-1}^N}-\esp{F\,|\,\h{k}^N}.
    \end{multline*}
    For $F\in \f{N}$, since the successive terms collapse, we get
    \begin{multline*}
      F-\esp{F}=\esp{F\,|\,\h{0}^N}-\esp{F\,|\, \h{N}^N}\\ =
      \sum_{k=1}^{N}D_k\,\esp{F\,|\,\h{k-1}^N}=\sum_{k=1}^{\infty}D_k\,\esp{F\,|\,\h{k-1}^N},
    \end{multline*}
    by the very definition of the gradient map. As in~\eqref{eq_gradient:3}, we
    can show that for any $N$,
    \begin{equation*}
      \esp{D_k\,\esp{F\,|\,\h{k-1}^N}\ D_l\,\esp{F\,|\,\h{l-1}^N}}=0, \text{ for } k\neq l.
    \end{equation*}
    Consider $F_N=\esp{F\,|\,\f{N}}$ and proceed as in the proof of
    Lemma~\ref{lchaos} to conclude.
  \end{proof}
  \begin{proof}[Proof of Corollary~\protect\ref{cor:poincare}]
    According to~\eqref{eq_gradient:3} and~\eqref{chaos2}, we have
    \begin{align*}
      \var(F)
        &=\esp{\left|\sum_{k\in A}D_k\,\esp{F|\f{k}}\right|^2}\\
        & =\esp{\sum_{k\in A}\Big|D_k\,\esp{F|\f{k}}\Big|^2}\\
        &=\esp{\sum_{k\in A}\Big|\esp{D_k\,F|\f{k}}\Big|^2}\\
        &\le \esp{\sum_{k\in A}\esp{|D_kF|^2|\f{k}}}
          =\esp{\sum_{k\in A}|D_kF|^2},
    \end{align*}
    where the inequality follows from then Jensen inequality.
  \end{proof}
  \begin{proof}[Proof of Theorem~\protect\ref{thm:cov1}]
    Let $F,G\in\DD$, the Clark formula entails
    \begin{align*}\label{cov_3}
      \cov(F,G)
      &=\esp{(F-\esp{F})(G-\esp{G})}\\
      &=\esp{\sum_{k,l\in A}D_k\esp{F\,|\,\f{k}}\ D_l\esp{G\,|\,\f{l}}}\\
      &=\esp{\sum_{k\in A}D_k\esp{F\,|\,\f{k}}\ D_k\esp{G\,|\,\f{k}}}\\
      &=\esp{\sum_{k\in A}D_kF\ D_k\esp{G\,|\,\f{k}}}
    \end{align*}
    where we have used~\eqref{eq_gradient:3} in the third equality and the
    identity $D_kD_k=D_k$ in the last one.
  \end{proof}
  \begin{proof}[Proof of Theorem~\protect\ref{thm:cov2}]
    Let $F,G\in L^2(E_A)$.
    \begin{align*}
      \cov(F,G)
      &=\esp{\sum_{k\in A}D_k\esp{F|\f{k}}D_k\esp{G|\f{k}}}\\
      &=\esp{\sum_{k\in A}D_k\esp{F|\f{k}}\left(-\int_0^{\infty}LP_t\esp{G|\f{k}}\dif t\right)}\\
      &=\int_0^{\infty}\esp{\sum_{k\in A}D_k\esp{F|\f{k}}\left(\sum_{l\in A}D_lP_t\esp{G|\f{k}}\dif t\right)}\\
      &=\int_0^{\infty}e^{-t}\esp{\sum_{k\in A}D_kF\,P_tD_k\esp{G|\f{k}}}\dif t
    \end{align*}
    when we have used the orthogonality of the sum, (\ref{OU-Grad}) and the
    $\f{k}$-measurability of $P_tD_k\esp{G|\f{k}}$ to get the last equality.
  \end{proof}
  \begin{proof}[Proof of Theorem~\protect\ref{thm:concentration}]
    Assume with no loss of generality that $F$ is centered. Apply~\eqref{cov_2}
    to $\theta F$ and~$e^{\theta F}$,
    \begin{align*}
      \theta \left|\esp{Fe^{\theta F}}\right|&=\theta\left|\esp{\sum_{k\in A}D_kF\
                                               D_k\esp{e^{\theta
                                               F}\,|\,\f{k}}}\right|\\
                                             &\le \theta \sum_{k\in A} \esp{ |D_kF|\ \Bigl| D_k\esp{e^{\theta
                                               F}\,|\,\f{k}}\Bigr|}.
    \end{align*}
    Recall that
    \begin{align*}
      D_k\esp{e^{\theta F}\,|\,\f{k}}&=\mathbf E'\left[\esp{e^{\theta
                                       F}\,|\,\f{k}} -\esp{e^{\theta F(X_{\neg
                                       k},X'_k)}\,|\,\f{k}}\right]\\
                                     &=\esp{ \mathbf E'\left[ e^{\theta F}- e^{\theta F(X_{\neg
                                       k},X'_k)}\right]\,\Bigl|\, \f{k}}\\
                                     &=\esp{ e^{\theta F}\, \mathbf E'\left[1- e^{-\theta \Delta_{k}F}\right]\,\Bigl|\, \f{k}}
    \end{align*}
    where $\Delta_k F=F- F(X_{\neg k},X'_k)$ so that $D_kF=\mathbf
    E'\left[\Delta_k F\right]$.

    Since $(x\mapsto 1-e^{-x})$ is concave, we get
    \begin{equation*}
      D_k\esp{e^{\theta F}\,|\,\f{k}}\le \esp{e^{{\theta
            F}}(1-e^{-\theta D_{k}F})\, |\, \f{k}} \le \theta  \ \esp{e^{{\theta
            F}}\,|D_{k}F|\, |\, \f{k}}.
    \end{equation*}
    Thus,
    \begin{equation*}
      \label{eq_gradient:8}
      \left|\esp{Fe^{\theta F}}\right|\le\theta\   \esp{e^{\theta
          F}  \sum_{{k=1}}^{\infty}\ |D_{k}F|\,\esp{|D_{k}F|\, |\, \f{k}}}\le M\,
      \theta \  \esp{e^{\theta F}}.
    \end{equation*}
    By Gronwall lemma, this implies that
    \begin{equation*}
      \esp{e^{\theta F}}\le \exp\left({\frac{\theta^2}{2}\, M}\right)\cdotp
    \end{equation*}
    Hence,
    \begin{equation*}
      \P(F-\esp{F}\ge x)=\P(e^{\theta(F-\esp{F})})\ge e^{\theta x})\le
      \exp({-\theta x+\frac{\theta^2}{2}\, M}).
    \end{equation*}
    Optimize with respect to $\theta$ gives $\theta_{\text{opt}}=x/M$, hence the
    result.
  \end{proof}
  \begin{proof}[Proof of Theorem~\protect\ref{thm:logSob}]
    We follow closely the proof of~\cite{Wu:2000lr} for Poisson process. Let
    $G\in L^2(E_A)$ be a positive random variable such that $DG\in L^{2}(A\times
    E_A)$. For any non-zero integer $n$, define $G_n=\min(\max(\frac{1}{n},G),n)$, for
    any $k$, $L_k=\esp{G_n|\f{k}}$ and $L_0=\esp{G_n}$. We have,
    \begin{align*}
      L_n\log L_n-L_0\log L_0
      &=\sum_{k=0}^{n-1}L_{k+1}\log L_{k+1}- L_{k}\log L_{k}\\
      &=\sum_{k=0}^{n-1}\log L_{k}(L_{k+1}-L_{k})+\sum_{k=0}^{n-1}L_{k+1}(\log L_{k+1}-\log L_{k})  .
    \end{align*}
    Note that $\left(\log L_{k}(L_{k+1}-L_{k}),\ k\ge 0\right)$ and
    $(L_{k+1}-L_{k},\, k\ge 0)$ are martingales, hence
    \begin{multline*}
      \esp{L_n\log L_n-L_0\log L_0}\\
      \begin{aligned}
        &=\esp{\sum_{k=0}^{n-1}L_{k+1}\log L_{k+1}-L_{k+1}\log L_k-L_{k+1}+L_k}\\
        &=\esp{\sum_{k=0}^{n-1}L_{k+1}\log L_{k+1}-L_{k}\log L_{k}-(\log L_{k}+1)( L_{k+1}-L_{k})}\\
        &=\esp{\sum_{k=0}^{n-1}\ell(L_k,\, L_{k+1}-L_k)},
      \end{aligned}
    \end{multline*}
    where the function $\ell$ is defined on $\Theta=\{(x,y)\in\R^2 : x>0,
    x+y>0\}$ by
    \begin{equation*}
      \ell(x,y) = (x+y)\log(x+y)-x\log x-(\log x+1)y.
    \end{equation*}
    Since $\ell$ is convex on $\Theta$, it comes from the Jensen inequality for
    conditional expectations that
    \begin{align*}
      \sum_{k=0}^{n-1}\esp{\ell(L_k,L_{k+1}-L_k)}
      &=\sum_{k=0}^{n-1}\esp{\ell(\esp{G_n\,|\,\f{k}},D_{k+1}\esp{G_n\,|\,\f{k+1}})}\\ 
      &=\sum_{k=1}^{n}\esp{\ell(\esp{G_n\,|\,\f{k-1}},\esp{D_{k}G_n\,|\,\f{k}})}\\ 
      &\le\sum_{k=1}^{n}\esp{\esp{\ell(\esp{G_n\,|\,\g{k}},D_kG_n)\,|\,\f{k}}}\\
      &=\sum_{k=1}^{n}\esp{\ell(\esp{G_n\,|\,\g{k}},D_kG_n)}\\
      &=  \sum_{k=1}^{\infty}\esp{\ell(\esp{G_n\,|\,\g{k}},D_kG_n)}.
    \end{align*}
    We know from~\cite{Wu:2000lr} that for any non-zero integer $k,$
    $\ell(\esp{G_n\,|\,\g{k}},D_kG_n)$ converges increasingly to
    $\ell(\esp{G\,|\,\g{k}},D_kG)$ $\P$-a.s., hence by Fatou Lemma,
    \begin{equation*}
      \esp{G\log G}-\esp{G}\log\esp{G}\le \sum_{k=1}^{\infty}\esp{\ell(\esp{G\,|\,\g{k}},D_kG)}.
    \end{equation*}
    Furthermore, for any $(x,y)\in \Theta$,
    \begin{math}
      \ell(x,y)\le {|y|^2}/{x},
    \end{math}
    then,
    \begin{equation*}\label{logsob}
      \esp{G\log G}-\esp{G}\log\esp{G}\le\sum_{k=1}^{\infty}\esp{\frac{|D_kG|^2}{\esp{G\,|\,\g{k}}}}\cdotp
    \end{equation*}
    The proof is thus complete.
  \end{proof}
  \begin{proof}[Proof of Theorem~\protect\ref{thm:helmholtz}]
    We first prove the uniqueness. Let $(\varphi,\, V)$ and $(\varphi',\, V')$
    two convenient couples. We have
    \begin{math}
      D_a(\varphi-\varphi')=V_a'-V_a
    \end{math}
    for any $a\in A$ and $\sum_{a\in A} D_a(V_a'-V_a)=0$, hence
    \begin{multline*}
      0= \esp{ (\varphi-\varphi')\sum_{a\in A} D_a(V_a'-V_a)}= \esp{\sum_{a\in
          A} D_a(\varphi-\varphi')(V_a'-V_a)}\\ = \esp{\sum_{a\in
          A}(V'_a-V_a)^2}.
    \end{multline*}
    This implies that $V=V'$ and $D(\varphi-\varphi')=0$. The Clark formula
    (Theorem~\ref{lchaos}) entails that
    $0=\esp{\varphi-\varphi'}=\varphi-\varphi'$.

    We now prove the existence. Since $\esp{D_a\varphi\, |\, \exv_a}=0,$ we can
    choose
    \begin{equation*}
      V_{a}=\esp{U_{a}\,|\,\mathcal{G}_{a}},
    \end{equation*}
    which implies \begin{math} D_a\varphi =D_aU_a,
    \end{math} and guarantees $\delta V=0$. Choose any ordering of the elements
    of $A $ and remark that, in view of~\eqref{eq_gradient:3},
    \begin{multline*}
      \esp{\left( \sum_{k=1}^\infty \esp{D_kU_k\, |\, \f{k}}
        \right)^2}=\esp{\left( \sum_{k=1}^\infty D_k\esp{U_k\, |\, \f{k}}
        \right)^2}\\=\esp{ \sum_{k=1}^\infty \Bigl(D_k\esp{U_k\, |\, \f{k}}
        \Bigr)^2} \le \sum_{k=1}^\infty \esp{|D_kU_k|^2}\le
      \|U\|_{\DD(l^2(A))}^2,
    \end{multline*}
    hence
    \begin{equation*}
      \varphi=\sum_{k=1}^\infty \esp{D_kU_k\, |\, \f{k}},
    \end{equation*}
    defines a square integrable random variable of null expectation, which
    satisfies the required property.
  \end{proof}
  \subsection{Proofs of Section \protect \ref{sec:dirichlet-structures}}
  \label{sec:p4}
  \begin{proof}[Proof of Theorem~\protect\ref{thm_Article-part1:1}]
    Starting from~\eqref{eq_Article-part1:3}, the terms with $\tau=0$ can be
    decomposed as
    \begin{equation*}
      e^{-2p_m^N} \sum_{m=1}^N \esp{ \left(F(\omega_{(m)}^N+\varepsilon_{\zeta_m^N})-
          F(\omega_{(m)}^N)\right)^2}\mu_m^N(1)+ R_0^N.
    \end{equation*}
    Since $F$ belongs to $\TVlip$,
    \begin{equation*}
      R_0^N \le   \sum_{m=1}^N\sum_{\ell=2}^{\infty} l^2
      \mu_m^N(l)\le c_1\, N(p^N)^2 \esp{(\text{Poisson}(p^N)+2)^2}\le c_2\, N (p^N)^2,
    \end{equation*}
    where the $c_1$ and $c_2$ are irrelevant constants. As $Np^N$ is bounded,
    $R_0^N$ goes to $0$ as $N$ grows to infinity. For the very same reasons, the
    sum of the terms of~\eqref{eq_Article-part1:3} with $\tau \ge 1$ converge to
    $0$, thus
    \begin{equation*}
      \lim_{N\to \infty}\mathcal E^{U_N}(F)=\lim_{N\to \infty}\sum_{m=1}^N
      e^{-2p_m^N}\ \esp{ \left(F(\omega_{(m)}^N+\varepsilon_{\zeta_m^N})-
          F(\omega_{(m)}^N)\right)^2}\, p_m^N.
    \end{equation*}
    Consider now the space $\fN_\bY^\zeta=\fN_\bY\times \{\zeta_k^N,\,
    k=1,\cdots,N\}$ with the product topology and probability measure
    $\tilde{\P}_N=\P_N\otimes \sum_k p_k^N\, \varepsilon_{\zeta_k^N}$. Let
    \begin{align*}
      \psi \, :\, \fN_\bY\times \{\zeta_k^N,\, k=1,\cdots,N\} &
                                                                \longrightarrow
                                                                E\\
      (\omega,\, \zeta) & \longmapsto
                          \Bigl(F(\omega-(\omega(\zeta)-1)\varepsilon_{\zeta})-F(\omega-\omega(\zeta)\varepsilon_{\zeta})\Bigr)^2.
    \end{align*}
    Then, we can write
    \begin{equation*}
      \sum_{m=1}^N\esp{ \left(F(\omega_{(m)}^N+\varepsilon_{\zeta_m^N})-
          F(\omega_{(m)}^N)\right)^2}\, p_m^N=\int_{\fN_\bY^\zeta}
      \psi(\omega, \zeta)\dif \tilde{\P}_N(\omega,\zeta).
    \end{equation*}
    Under $\tilde{\P}_N$, the random variables $\omega$ and $\zeta$ are
    independent. Equation~\eqref{eq_Article-part1:2} means that the marginal
    distribution of $\zeta$ tends to $\bM$ (assumed to be a probability measure
    at the very beginning of this construction). Moreover, we already know that
    $\P_N$ converges in distribution to $\P$. Hence, $\tilde{\P}_N$ tends to
    $\P\otimes \bM$ as $N$ goes to infinity. Since $F$ is in $\TVlip$, $\psi$ is
    continuous and bounded, hence the result.
  \end{proof}

\begin{proof}[Proof of Theorem~\protect\ref{thm:donsker}]
  For $F\in \DD_{B}\cap \HC$, in view of~\eqref{eq_Article-part1:5}, we have
  \begin{multline*}
    F(\omega^N)-F(\omega^N_{(k)}+M'_k\,h_k^N)\\=(M_k-M'_k)\,\langle \nabla
    F(\omega_{(k)}^N),\, h_k^N\rangle_H+\frac{M_k-M'_k}{\sqrt{N}}\
    \varepsilon(\omega_{(k)}^N,h_k^N).
  \end{multline*}
  Hence,
  \begin{multline*}
    \sum_{k=1}^N \esp{\Bigl(F(\omega^N)-\mathbf
      E'\left[F(\omega^N_{(k)}+M'_k\,h_k^N)\right]\Bigr)^2}
    \\
    = \sum_{k=1}^N \esp{\Bigl(M_k\,\langle \nabla F(\omega_{(k)}^N),\,
      h_k^N\rangle_H+\mathbf E'\left[\frac{M_k-M'_k}{\sqrt{N}}\
        \varepsilon(\omega_{(k)}^N,h_k^N)\right] \Bigr)^2}
    \\
    =\sum_{k=1}^N \esp{\langle \nabla F(\omega_{(k)}^N),\,
      h_k^N\rangle_H^2}+\text{Rem},
  \end{multline*}
  and
  \begin{equation*}
    \text{Rem}\le \frac{c}{N}\sum_{k=1}^N \esp{
      \varepsilon(\omega_{(k)}^N,h_k^N)^2}\xrightarrow{N\to \infty}0,
  \end{equation*}
  by the C\'esaro theorem. It follows that $\mathcal E^{U_N}(F)$ has the same
  limit as
  \begin{equation*}
    \sum_{k=1}^N \esp{\langle                                                             \nabla
      F(\omega_{(k)}^N),\,
      h_k^N\rangle_H^2}.
  \end{equation*}
  As $N$ goes to infinity, we add more and more terms to the random walk, so
  that the influence of one particular term becomes negligible. The following
  result is well known~\cite[Proposition 3]{bouleau_theoreme_2005}: For any $k\in\{1,\dots,N\}$,
  for any bounded $\psi$ and $\varphi$,
  \begin{equation*}
    \esp{\psi(M_k)\varphi(\omega^N)}\xrightarrow{N\to \infty} \esp{\psi(M_k)}\esp{\varphi(\omega)}.
  \end{equation*}
  Since $\|\nabla F\|_H$ belongs to $L^\infty$ and $\|h_k^N\|_\infty $ tends to
  $0$, this entails that for any $k$,
  \begin{multline*}
    \lim_{N\to \infty}\esp{\langle \nabla F(\omega_{(k)}^N),\,
      h_k^N\rangle_H^2}= \lim_{N\to \infty}\esp{\langle \nabla F(\omega^N),\,
      h_k^N\rangle_H^2}\\ =\lim_{N\to \infty}\esp{\|\pi_{V_N}\nabla
      F(\omega^N)\|^2_H},
  \end{multline*}
  where $\pi_{V_N}$ is the orthogonal projection in $H$ onto
  $\text{span}\{h_k^N,\, k=1,\cdots,N\}$. We conclude by dominated convergence.
\end{proof}

\subsection{Proofs of Section \protect \ref{sec:appl-perm}}
\begin{proof}[Proof of Theorem~\protect\ref{thm_gradient:1}]
  Take care that in the argument of $h$, all the sets are considered as ordered:
  When we write $B\cup C$, we implicitly reorder its elements, for instance
  \begin{equation*}
    h(X_{\{1,3\}\cup\{2\}})=h(X_1,X_2,X_3).
  \end{equation*}
  Apply the Clark formula,
  \begin{align*}
    U_n-\theta
    &=\binom{n}{m}^{-1}\sum_{A\in ([n],m)}\sum_{B\subset A} \binom{m}{|B|}^{-1}\frac1{|B|}\sum_{b\in B}D_b\esp{h(X_A)\, |\, X_B}\\
    &=\binom{n}{m}^{-1}\sum_{B\subset [n]}  \binom{m}{|B|}^{-1} \frac1{|B|} \sum_{b\in B}\sum_{\substack{A\supset B\\ A\in ([n],m)}}D_b\esp{h(X_A)\, |\, X_B}\\
    &=\binom{n}{m}^{-1}\sum_{B\subset [n]}  \binom{m}{|B|}^{-1} \frac1{|B|} \sum_{b\in B}\ \sum_{C\in ([n]\backslash B, \, m-|B|)}D_b\esp{h(X_{B\cup C})\, |\, X_B}.
  \end{align*}
  It remains to prove that
  \begin{multline}\label{hoeffding}
    \sum_{k=1}^{m}\binom{m}{k}H^{(k)}_n\\
    =\binom{n}{m}^{-1}\sum_{B\subset [n],|B|\le m} \binom{m}{|B|}^{-1}
    \frac1{|B|} \sum_{b\in B}\ \sum_{C\in ([n]\backslash B, \,
      m-|B|)}D_b\esp{h(X_{B\cup C})\, |\, X_B}.
  \end{multline}
  for any integer $n$. For $n=1$, it is straightforward that
  \begin{align*}
    g_1(X_1)=h(X_1)-\theta=D_1\esp{h(X_{1})|X_{1}}.
  \end{align*}
  Assume the existence of an integer $n$ such that (\ref{hoeffding}) holds for
  any set of cardinality $n$. In particular, for any $l\in[n+1]$
  \begin{multline*}
    \sum_{k=1}^{m}\binom{m}{k}H^{(k)}_{A_l}\\
    =\binom{n}{m}^{-1}\sum_{B\subset [A_l],|B|\le m} \binom{m}{|B|}^{-1}
    \frac1{|B|} \sum_{b\in B}\ \sum_{C\in ([A_l]\backslash B, \,
      m-|B|)}D_b\esp{h(X_{B\cup C})\, |\, X_B},
  \end{multline*}
  where $A_l=[n+1]\backslash\{l\}$. Let $m$ such that $m\le n$. Then,
  \begin{multline*}
    \sum_{k=1}^{m}\binom{m}{k}H^{(k)}_{n+1}\\
    \shoveleft{=\sum_{k=1}^{m}\binom{m}{k}\binom{n+1}{k}^{-1}\frac{1}{n+1-k}\sum_{l=1}^{n+1}\sum_{B\in([A_l],k)}g_k(X_B)}\\
    \shoveleft{=\frac{1}{n+1}\sum_{l=1}^{n+1}\sum_{k=1}^{m}\binom{m}{k}\binom{n}{k}^{-1}\sum_{B\in([A_l],k)}g_k(X_B)}\\
    \shoveleft{=\frac{1}{n+1}\sum_{l=1}^{n+1}\binom{n}{m}^{-1}}\\
    \shoveright{\times\sum_{B\subset [A_l],|A_l|\le m} \binom{m}{|B|}^{-1} \frac1{|B|} \sum_{b\in B}\ \sum_{C\in ([A_l]\backslash B, \, m-|B|)}D_b\esp{h(X_{B\cup C})\, |\, X_B}}\\
    \shoveleft{=\frac{n+1-m}{n+1}\binom{n}{m}^{-1}}\\
    \shoveright{\times\sum_{B\subset [n+1],|B|\le m} \binom{m}{|B|}^{-1} \frac1{|B|} \sum_{b\in B}\ \sum_{C\in ([n+1]\backslash B, \, m-|B|)}D_b\esp{h(X_{B\cup C})\, |\, X_B}}\\
    \shoveleft{=\binom{n+1}{m}^{-1}}\\\times\sum_{B\subset [n+1],|B|\le m}
    \binom{m}{|B|}^{-1} \frac1{|B|} \sum_{b\in B}\ \sum_{C\in ([n+1]\backslash
      B, \, m-|B|)}D_b\esp{h(X_{B\cup C})\, |\, X_B},
  \end{multline*}
  where we have used in the first line that each subset $B$ of $[n+1]$ of
  cardinality $k$ appears in $n+1-k$ different subsets $A_l$ (for
  $l\in[n+1]\backslash B$), and in the same way, in the penultimate line, that
  each subset ${B\cup C}$ of $[n+1]$ of cardinality $m$ appears in $n+1-m$
  different subsets $A_l$ (for $l\in[n+1]\backslash B\cup C$). Eventually, the
  case $m=n+1$ follows from
  \begin{align*}
    \sum_{k=1}^{n+1}\sum_{B\in([n+1],k)}g_k(X_B)
    &=h(X_{[n+1]})-\theta\\
    &=\sum_{B\subset[n+1]}\binom{n+1}{|B|}^{-1}\frac{1}{|B|}\sum_{b\in B}D_b\esp{h(X_{[n+1]})\,|\,X_B},
  \end{align*}
  by applying the Clark formula to $h$.
\end{proof}

\begin{proof}[Proof of Theorem~\protect\ref{U_k}]
  By the previous construction, for
  \begin{equation*}
    i=(i_1,\cdots,i_N)\in (I_k=k)\cap \bigcap_{m=k+1}^N (I_m\neq k),
  \end{equation*}
  the permutation $\sigma=\Gamma(i)$ admits $k$ as a fixed point. Hence,
  \begin{equation*}
    \left\{  (I_k=k)\cap \bigcap_{m=k+1}^N (I_m\neq k) \right\}\, \subset\  (\tilde U^N_k=1).
  \end{equation*}
  As both events have cardinality $(N-1)!$, they do coincide.
  The values of $p_k$ and $\alpha_k$ are easily computed since the random
  variables $(I_m,\, k\le m \le N)$ are independent.
  According to Theorem~\ref{lchaos:reverse},
  \begin{multline*}
    \tilde U^N_{k}=\esp{\tilde U^N_{k}}+ \sum_{l=1}^{N}D_{l}\esp{\tilde
      U_{k}\,|\,\h{l-1}}\\
    =\esp{\tilde U^N_{k}}+ \sum_{l=1}^{N}\esp{\tilde
      U^N_{k}\,|\,\h{l-1}}-\esp{\tilde U^N_{k}\,|\,\h{l}}.
  \end{multline*}
  Since $\tilde U^N_{k}\in \h{k-1}$, $D_{l}\esp{\tilde U_{k}\,|\,\h{l-1}}=0$ for
  $l<k$. For $l=k$, we get
  \begin{multline*}
    \esp{\car_{(I_k=k)}\!\prod_{{m=k+1}}^{N}\car_{(I_m\neq k)} \,
      |\, I_{k},\, I_{k+1},\cdots}\\
    \shoveright{- \esp{\car_{(I_k=k)}\!\prod_{{m=k+1}}^{N}\car_{(I_m\neq k)} \,
        |\, I_{k+1},\, I_{k+2},\cdots}}\\
    =\Bigl(\car_{(I_k=k)}-\P_{k}(\{k\})\Bigr)\prod_{{m=k+1}}^{N}\car_{(I_m\neq
      k)}.
  \end{multline*}
  For $l=k+1$,
  \begin{multline*}
    \esp{\car_{(I_k=k)}\!\prod_{{m=k+1}}^{N}\car_{(I_m\neq k)} \,
      |\, I_{k+1},\, I_{k+2},\cdots}\\
    \shoveright{- \esp{\car_{(I_k=k)}\!\prod_{{m=k+1}}^{N}\car_{(I_m\neq k)} \,
        |\, I_{k+2},\, I_{k+3},\cdots}}\\
    \begin{aligned}
      &=tp_k\Bigl(\car_{(I_{k+1}\neq
        k)}-\P_{k+1}(\{k\}^c)\Bigr)\prod_{{m=k+2}}^{N}\car_{(I_m\neq k)}\\
      &=-tp_k \Bigl(\car_{(I_{k+1}=
        k)}-\P_{k+1}(\{k\})\Bigr)\prod_{{m=k+2}}^{N}\car_{(I_m\neq k)}.
    \end{aligned}
  \end{multline*}
  The subsequent terms are handled similarly and the result follows.
\end{proof}
\begin{proof}[Proof of Theorem~\protect\ref{thm:decompositionC1}]
  By the very definition of $\tilde C_1$, we have
  \begin{equation}
    \label{eq:3}    \tilde C_1=\esp{\tilde C_1}+\sum_{k=1}^N \sum_{l= k}^N D_l \esp{\tilde U^N_k\,|\,\h{l-1}}.
  \end{equation}
  For $k=l$, $\esp{\tilde U^N_k\,|\,\h{l-1}}=\tilde U^N_k$ and for $l>k$,
  \begin{align*}
    \esp{\tilde U^N_k\,|\,\h{l-1}}&=\frac{t}{t+k-1}\left( 1-\frac{1}{t+k} \right)\ldots \left( 1-\frac{1}{t+l-2} \right)\prod_{m=l}^N\car_{(I_m\neq k)}\\
                                  &=\frac{t}{t+l-2}\prod_{m=l}^N\car_{(I_m\neq k)}.
  \end{align*}
  It is straightforward that $l>k$,
  \begin{align*}
    D_l \left( \prod_{m=l}^N\car_{(I_m\neq k)} \right)&=\left( \car_{(I_l\neq k)}-(1-\frac{1}{t+l-1}) \right)\prod_{m=l+1}^N\car_{(I_m\neq k)}\\
                                                      &=-\left( \car_{(I_l= k)}-\frac{1}{t+l-1}\right)\prod_{m=l+1}^N\car_{(I_m\neq k)}.
  \end{align*}
  The result then follows by direct computations.
\end{proof}
\begin{proof}[Proof of Theorem~\protect\ref{thm:varianceC1}]
  Recall that for $j\neq l$, $D_l \esp{\tilde U^N_k\,|\,\h{l-1}}$ and $D_j
  \esp{\tilde U^N_m\,|\,\h{j-1}}$ are orthogonal in $L^2$. In view of
  \eqref{eq:3}, according to the integration by parts formula, we have
  \begin{multline*}
    \begin{aligned}
      \var{[\tilde C_1]}
      &=\sum_{k=1}^N \sum_{m=1}^N \sum_{l= k}^N \sum_{j= m}^N \esp{ D_l \esp{\tilde U^N_k\,|\,\h{l-1}}D_j \esp{\tilde U^N_m\,|\,\h{j-1}}}\\
      &=\sum_{k=1}^N \sum_{m=1}^N\sum_{l=k\vee m}^N \esp{ D_l \esp{\tilde U^N_k\,|\,\h{l-1}}D_l \esp{\tilde U^N_m\,|\,\h{l-1}}}\\
      &=2\sum_{k=1}^N \sum_{m=k+1}^N\sum_{l=m}^N \esp{ U^N_k\,D_l \esp{\tilde
          U^N_m\,|\,\h{l-1}}}
    \end{aligned}  \\
    +\esp{\sum_{k=1}^N \sum_{l=k}^N \tilde U^N_k\,D_l \esp{\tilde
        U^N_k\,|\,\h{l-1}}}.
  \end{multline*}
  Then, for~$l\ge m>k$,
  \begin{multline*}
    \esp{U^N_k\,D_l \esp{\tilde U^N_m\,|\,\h{l-1}}}\\
    \begin{aligned}
      &=-\frac{t}{t+l-2}\esp{\car_{(I_k= k)}\prod_{p=k+1}^N\car_{(I_p\neq k)}\left( \car_{(I_l= m)}-\frac{1}{t+l-1}\right)\prod_{j=l+1}^N\car_{(I_j\neq m)}}\\
      &=-\frac{t\, \P_k(\{k\})}{t+l-2}\left(\P_l(\{m\})-\frac{1}{t+l-1}\right)\esp{\prod_{p=k+1}^{l-1}\car_{(I_p\neq k)}}\esp{\prod_{p=l+1}^N\car_{(I_p\notin \{k,m\})}}\\
      &=0,
    \end{aligned}
  \end{multline*}
  since, for any $l\ge m>k$
  \begin{equation*}
    \esp{\car_{(I_l= m)}\car_{(I_l\neq k)}}=\esp{\car_{(I_l= m)}}=\P_l(\{m\})=\frac{1}{t+l-1}.
  \end{equation*} 
  Furthermore, for $l>k$,
  \begin{multline*}
    \esp{\tilde U^N_k\,D_l \esp{\tilde U^N_k\,|\,\h{l-1}}}\\
    \begin{aligned}
      &=-\frac{t}{t+l-2}\esp{\car_{(I_k= k)}\prod_{p=k+1}^N\car_{(I_p\neq k)}\left( \car_{(I_l= k)}-\frac{1}{t+l-1}\right)\prod_{p=l+1}^N\car_{(I_p\neq k)}}\\
      &=\frac{t}{(t+l-1)(t+l-2)}\P_k(\{k\})\esp{\prod_{p=k+1}^N\car_{(I_p\neq k)}}\\
      &=\frac{t^2}{(t+l-1)(t+l-2)(t+N-1)},
    \end{aligned}
  \end{multline*}
  as $\prod_{p=k+1}^N\car_{(I_p\neq k)}\car_{(I_l= k)}=0$, for $l>k$. Finally,
  for $l=k$, we get
  \begin{multline*}
    \esp{\tilde U^N_k\,D_l \esp{\tilde U^N_k\,|\,\h{l-1}}}\\
    \begin{aligned}
      &=\esp{\car_{(I_k= k)}\prod_{p=k+1}^N\car_{(I_p\neq k)}\left( \car_{(I_k= k)}-\frac{t}{t+k-1}\right)\prod_{p=k+1}^N\car_{(I_p\neq k)}}\\
      &=\left(\frac{t}{t+k-1}-\frac{t^2}{(t+k-1)^2}\right)\frac{t+k-1}{t+N-1}\\
      &=\frac{t(k-1)}{(t+k-1)(t+N-1)}\cdotp
    \end{aligned}
  \end{multline*}
  It follows that
  \begin{multline*}
    \var{[\tilde C_1]}\\
    \begin{aligned}
      &=\frac{t^2}{t+N-1}\sum_{k=1}^N\sum_{l=k+1}^N\frac{1}{(t+l-1)(t+l-2)}+\frac{t}{t+N-1}\sum_{k=1}^N\frac{k-1}{t+k-1}\\
      &=\frac{t}{t+N-1}\left(\frac{Nt}{t+N-1}+N-2t\,\sum_{k=1}^{N}\frac{1}{t+k-1}\right).
    \end{aligned}
  \end{multline*}
  The proof is thus complete.
\end{proof}
\begin{proof}[Proof of Theorem~\protect{\ref{thm:3.1Gaussianbis}}]
  We have to compute
  \begin{equation*}
    \sup_{\vp\in \mathcal F}\esp{\vp'(F)-F\vp(F)},
  \end{equation*}
  where $\mathcal F$ is the set of twice differentiable functions with second order
  derivative bounded by $2$. Since $F$ is centered
  \begin{equation*}
    \esp{F\vp(F)}=\esp{LL^{-1}F\, \vp(F)}=\sum_{a\in  A}\esp{(-D_{a}L^{-1})F\, D_{a}\vp(F)}.
  \end{equation*}
  The trick is to use the Taylor expansion taking the reference point to be
  $X'_{\neg a}$ instead of  $X_{A}$. This yields
  \begin{equation*}
    D_{a}\vp(F)=\espp{\vp(F(X_{A}))-\vp(F(X'_{\neg
        a},X'_{a}))}=\vp'(F(X'_{\neg a}))D_{a}F+R,
  \end{equation*}
  where
  \begin{equation*}
    R=\frac{1}{2}\int_{0}^{1}\espp{\vp''\Bigl(\theta F(X'_{\neg a})
      +(1-\theta)F(X_{A}) \Bigr)
    \Bigl(F(X_{A})-F(X'_{\neg a })\Bigr)^{2}}\dif \theta.
  \end{equation*}
 Hence
\begin{multline*}
  \esp{\vp'(F)-F\vp(F)}\\=\esp{\vp'(F)-\sum_{a\in A}\vp'(F(X'_{\neg a}))\ D_{a}F (-D_{a}L^{-1})F }\\
  +\sum_{a\in A} \esp{R\ (-D_{a}L^{-1})F }.
\end{multline*}
The rightmost term of the the latter equation easily yields  the rightmost
of~\eqref{eq_gradient_spa_v2:17}. Since $\|\vp''\|_{\infty}<2$, it is clear that
$\vp'$ belongs to $\Lip_{2}$ hence the formulation of the distance with a supremum.
\end{proof}
\begin{proof}[Proof of Corollary~\protect{\ref{thm:lyapounov}}]
  Without loss of generality, we can assume that $X_{i}$ is centered for any $i\ge 1$.
  Remark that
  \begin{equation*}
    D_{j}X_{k}=
    \begin{cases}
      0 & \text{ if } j\neq k,\\
      X_{k}& \text{ if } j=k.
    \end{cases}
  \end{equation*}
  Hence $LY_{n}=Y_{n}$ and $Y_{n}=L^{-1}Y_{n}$. According to
  Theorem~\ref{thm:3.1Gaussianbis},
  \begin{multline*}
    \kappa_{\FF}(\P, \P_{Y_{n}})\le  \sup_{\psi\in\Lip_{2}} \esp{\psi(F)-\frac{1}{s_{n}^{2}}\sum_{i\in A}\psi\Bigl(F\bigl(Y_{n}-\frac{X_{i}-X'_{i}}{s_{n}}\bigr)\Bigr)
    X_{i}^{2}}\\
    +  \frac{1}{s_{n}^{3}}\sum_{j=1}^{n}\esp{\int_{E_{A}}\bigl(X_{i}-x\bigr)^{2}
      \dif \P_{i}(x) \ |X_{i}|}.
  \end{multline*}
  By independence, since $\psi$ is $2$-Lipschitz continuous,
  \begin{multline*}
    \left|  \esp{\psi(F)-\frac{1}{s_{n}^{2}}\sum_{i\in A}\psi\Bigl(F(Y_{n}-\frac{X_{i}-X'_{i}}{s_{n}})\Bigr)
    X_{i}^{2}}\right| \\= \left|\frac{1}{s_{n}^{2}}\sum_{i\in A}\sigma_{i}^{2}\,
    \esp{\psi(F)-\psi\Bigl(F(Y_{n}-\frac{X_{i}-X'_{i}}{s_{n}})\Bigr)}\right|\\
  \le \frac{2}{s_{n}^{3}}\,\sum_{i\in A}\sigma_{i}^{2}\,
  \esp{|X_{i}-X'_{i}|}\le \frac{2\sqrt{2}}{s_{n}^{3}}\,\sum_{i\in A}\sigma_{i}^{3}.
\end{multline*}
Moreover,
\begin{multline*}
  \esp{\int_{E_{A}}\Bigl(X_{i}-x\Bigr)^{2}
      \dif \P_{i}(x) \ |X_{i}|} =\esp{|X_{i}|^{3}}+\sigma^{2}\esp{|X_{i}|}\\ \le
    \esp{|X_{i}|^{3}}+\sigma^{3}\le 2\, \esp{|X_{i}|^{3}}
  \end{multline*}
  according to the H\"older inequality.  Hence the result.
\end{proof}
\begin{proof}[Proof of Theorem~\protect{\ref{thm:3.1Gamma}}]
  According to the principle of the Stein method, we have to estimate
  \begin{equation}\label{eq_gradient_spa_v2:11}
    \esp{\frac{1}{\lambda}\left( \vp(F)+\frac{r}{\lm} \right)-F\vp'(F)},
  \end{equation}
where $\varphi $ and its derivatives satisfy \eqref{eq_gradient_spa_v2:16}.
For any $a\in A$, thanks to the Taylor expansion, 
\begin{equation}\label{eq_gradient_spa_v2:6}
  -D_a\vp(F)
=\espp{\vp(F(X^{\neg a},X'_a))-\vp(F(X))}
=-\vp'(F)D_aF+R,
\end{equation}
where
\begin{multline}\label{eq_gradient_spa_v2:9}
  R
  =\frac{1}{2}\int_{0}^{1}(1-\theta)\\ \times \espp{\vp''\Big((1-\theta) F(X)+\theta
  F(X^{\neg a},X'_a)\Big)\Big(F(X)- F(X^{\neg a},X'_a)\Big)^2}\dif \theta
\end{multline}
According to \eqref{IPP} and to the definition of $L$,
\begin{multline}\label{eq_gradient_spa_v2:5}
\esp{F\vp(F)}
=\esp{LL^{-1}F\,\vp(F)}
=\esp{-\delta(DL^{-1}F)\vp(F)}\\
=\esp{\langle D\vp(F),-DL^{-1}F\textgreater_{L^{2}(A)}}.
\end{multline}
Plug \eqref{eq_gradient_spa_v2:6} into \eqref{eq_gradient_spa_v2:5}:
\begin{multline*}
  \esp{\langle D\vp(F),-DL^{-1}F\textgreater_{L^{2}(A)}}\\
  \begin{aligned}
&=-\sum_{a\in A} \esp{D_a\vp(F)D_a(L^{-1}F)}\\
&=-\sum_{a\in A} \esp{\vp'(F)D_aFD_a(L^{-1}F)}+\sum_{a\in A} \esp{R\ D_a(L^{-1}F)}\\
&=\esp{\vp'(F)\langle DF,-DL^{-1}F\textgreater_{L^{2}(A)}}+\esp{\langle R,-DL^{-1}F\textgreater_{L^{2}(A)}}    .
  \end{aligned}
\end{multline*}
Then,
\begin{multline*}
  \Big|\esp{\frac{1}{\lm}(F+\frac{r}{\lm})\vp'(F)-F\vp(F)}\Big|\\
\pp\left|\esp{\vp'(F)\ \Bigl( \frac{1}{\lm}(F+\frac{r}{\lm}) -\langle DF,-DL^{-1}F\textgreater_{L^{2}(A)}\Bigr)}\right|\\ +\left|\esp{\langle R,-DL^{-1}F\textgreater_{L^{2}(A)})}\right|=B_{1}+B_{2}.
\end{multline*}
Since $\vp'$ is bounded, we get
\begin{equation*}
B_{1}\le  \|\vp'\|_{\infty}\esp{\Bigl| \frac{1}{\lm}(F+\frac{r}{\lm}) -\langle DF,-DL^{-1}F\textgreater_{L^{2}(A)}\Bigr|}
\end{equation*}
and from \eqref{eq_gradient_spa_v2:9}, we deduce that 
\begin{equation*}
B_{2}\pp  \|\vp''\|_{\infty}\,\sum_{a\in A}\esp{|D_aF|^2|D_aL^{-1}F|}.
\end{equation*}
The proof follows from~\eqref{eq_gradient_spa_v2:11} and \eqref{eq_gradient_spa_v2:16}.
\end{proof}
\begin{proof}[Proof of Theorem~\protect\ref{gamma}]
  For any $a\in A$,
\begin{equation*}
D_a(X_iX_j)
=
\begin{cases}
  X_aX_j &\text{ if } a=i\\
  X_iX_a &\text{ if } a=j \\
  0 & \text{ otherwise.}
\end{cases}
\end{equation*}
Then, 
\begin{equation*}
D_aF=\sum_{(i,a)\in A^{\neq}}f(i,a)X_iX_a+\sum_{(j,a)\in A^{\neq}}f(a,j)X_aX_j=2\sum_{(i,a)\in A^{\neq}}f(i,a)X_iX_a,
\end{equation*}
so that
\begin{equation*}
LF=-\sum_{a\in A}D_aF=-2F \quad\text{and}\quad L^{-1}F=-\frac{1}{2}F.
\end{equation*}
With our notations, the first term of the right-hand-side
of~\eqref{eq_gradient_spa_v2:7} becomes
\begin{equation}\label{eq_gradient_spa_v2:8}
\esp{\left|2F +2\nu-2\sum_{a\in A}\sum_{(i,j)\in A^{2}}f(i,a)f(j,a)X_a^2X_{i}X_j\right|}\le \sum_{i=1}^{2}A_{i},
\end{equation}
where
\begin{align*}
A_{1}&=2\,\esp{\Big|\sum_{(i,a)\in A^{2}}f^{2}(i,a)(X_a^2X_i^2-1)\Big|}, \notag\\
A_{2}&=2\,\esp{\left|F -\sum_{a\in A}\sum_{(i,j)\in A^{\neq}}f(i,a)f(j,a)X_a^2X_iX_j\right|}.\notag
\end{align*}
We first control $A_{1}$. According to the Cauchy-Schwarz inequality, \begin{multline*}
A_{1}^{2}\le 4
\,\esp{\sum_{(i,a)\in  A^{2}}\sum_{(j,c)\in A^{2}}f^{2}(i,a)f^{2}(j,c)(X_a^2X_i^2-1)(X_c^2X_j^2-1)}\\
\pp 4(A_{11}+A_{12}),
\end{multline*}
where
\begin{align*}
A_{11}
&=\esp{\sum_{(i,a)\in A^{2}}f^{4}(i,a)(X_a^2X_i^2-1)^2},
\\
A_{12}
&=\,\esp{\sum_{a\in A}\sum_{(i,j)\in A^{\neq}}f^{2}(i,a)f^{2}(j,a)\,(X_a^2X_i^2-1)(X_a^2X_j^2-1)},
\end{align*}
by orthogonality of the $X_i$'s.
On the one hand, 
\begin{multline}\label{eq_gradient_spa_v3:2}
A_{11}
\pp \sum_{(i,a)\in A^{2}}f^{4}(i,a)\esp{\Big(X_a^2X_i^2-1\Big)^2}\\
=\big(\esp{X_1^4}^2-1\big)\sum_{(i,a)\in A^{2}}f^{4}(i,a).
\end{multline}
On the other hand,
\begin{align}
A_{12}
&=\esp{\sum_{(i,j)^{\neq}\in A^{2}}\sum_{a\in A}f^{2}(i,a)f^{2}(j,a)(X_a^2X_i^2-1)(X_a^2X_j^2-1)}\notag\\
&\pp \sum_{(i,j)^{\neq}\in A^{2}}\sum_{a\in A}f^{2}(i,a)f^{2}(j,a)\esp{(X_a^2X_i^2-1)(X_a^2X_j^2-1)}\notag\\
&=\big(\esp{X_1^4}-1\big) \sum_{(i,a)\in A^{2}}f^{2}(i,a)\sum_{j\neq i}f^{2}(j,a) \notag \\
&\le \big(\esp{X_1^4}-1\big)\  \|f\star_{2}^{1} f\|_{L^{2}(A)}^{2}.\label{eq_gradient_spa_v3:3}
\end{align}
In a similar way, $A_{2}\pp A_{21}+A_{22}$, where
\begin{align*}
A_{21}
&=2\,\esp{\Bigg|\sum_{(i,j)\in A^{\neq}}f(i,j)X_iX_j-\sum_{(i,j)\in A^{\neq}}\sum_{a\in A}f(i,a)f(j,a)X_iX_j\Bigg|},
\\
A_{22}
&=2\,\esp{\Bigg|\sum_{(i,j)\in A^{\neq}}\sum_{a\in A}f(i,a)f(j,a)X_iX_j\,\Bigl(X_{a}^{2}-\esp{X_a^2}\Bigr)\Bigg|}.
\end{align*}
As above,  
\begin{multline}\label{eq_gradient_spa_v3:4}
  A_{21}^{2}\le 4 \,\esp{\left(\sum_{(i,j)\in A^{\neq}}\Big(f(i,j)-\sum_{a\in
        A}f(i,a)f(j,a)\Big)X_iX_j  \right)^2}\\
=4 \, \|f-f\star_1^1 f\|_{2}^2.
\end{multline}

Furthermore, according to Cauchy-Schwarz inequality and by independence, we have
\begin{align}
A_{22}
&\pp 2\,\sum_{(i,j)\in A^{\neq}}\esp{|X_iX_j|\Big|\sum_{a\in A}f(i,a)f(j,a)(X_a^2-1)\Big|}\notag\\
&\pp 2\,\esp{\Big(\sum_{(i,j)\in A^{\neq}}\sum_{a\in A}f(i,a)f(j,a)(X_a^2-1)\Big)^2}^{1/2}\notag\\
  &\pp 2\,\left(\sum_{(i,j)\in A^{\neq}} \sum_{a\in A}f(i,a)^2f(j,a)^2\esp{X_a^4-1} \right)^{1/2}\notag\\
  &\le2  \big(\esp{X_1^4}-1\big)^{1/2}\  \|f\star_{2}^{1} f\|_{L^{2}(A)}.\label{eq_gradient_spa_v3:5}
\end{align}
The remainder term is given by 
\begin{equation*}
A_3=\sum_{a\in A}\esp{\int_{E_{A}}\Bigl( F(X_{A})-F(X_{A\neg a};x)\Bigr)^{2}
	\dif \P_{a}(x)\ |D_aL^{-1}F|}.
\end{equation*}
Once again, using the orthogonality, we have
\begin{align*}
G_a(X_A)&=\int_{E_{A}}\Bigl( F(X_{A})-F(X_{A\neg a};x)\Bigr)^{2}
\dif \P_{a}(x)\\
&=4\,\espp{\Big(\sum_{i\in A}f(i,a)X_iX_a-\sum_{i\in A}f(i,a)X_iX'_a\Big)^2}\\
&=4\,\espp{(X_a-X'_a)^2\Big(\sum_{i\in A} f(i,a)X_i\Big)^2}\\
&=4\,\Big(\sum_{i\in A} f(i,a)X_i\Big)^2\,\espp{(X_a-X'_a)^2}\\
&=4\,\Big(\sum_{i\in A} f(i,a)X_i\Big)^2\,\big(X_a^2+1\big).
\end{align*}
Thus,
\begin{multline}\label{eq_gradient_spa_v3:6}
\esp{\sum_{a\in A}G_a(X_A)^2}
=16\,\esp{\sum_{a\in A}\Big(\sum_{i\in A} f(i,a)X_i\Big)^4\,\big(X_a^2+1\big)^2}\\
\shoveleft{=16\,\big(\esp{X_1^4}+3\big)\esp{X_1^4}\,\sum_{a\in A}\sum_{i\in A}f^{4}(i,a)}\\
\shoveright{+96\,\big(\esp{X_1^4}+3\big)\sum_{a\in A}\sum_{(i,j)\in
  A^{\neq}}f^{2}(i,a)f^{2}(j,a)}\\
\le 16\,\big(\esp{X_1^4}+3\big)^{2}\sum_{a\in A}\sum_{i\in A}f^{4}(i,a)+ 96\,\big(\esp{X_1^4}+3\big)\|f\star_{2}^{1} f\|_{L^{2}(A)}^{2}.
\end{multline}
Moreover,
\begin{align}
  \sum_{a\in A}\esp{|D_{a}L^{-1}F|^{2}}&= \frac14\sum_{a\in A}\esp{|D_{a}F|^{2}}\notag\\
                                       &=\sum_{a\in A}\esp{\left( \sum_{(i,a)\in A^{\neq} }f(i,a)X_{i}X_{a} \right)^{2}}\notag\\
  &=\sum_{(i,a)\in A^{\neq} }f^{2}(i,a)=\nu.\label{eq_gradient_spa_v3:7}
\end{align}
Combine \eqref{eq_gradient_spa_v3:2}--\eqref{eq_gradient_spa_v3:7} to obtain~\eqref{eq_gradient_spa_v3:1}.
\end{proof}

\noindent\textbf{Acknowledgments:} The authors are indebted to the anonymous referees and  to the AE for many
  helpful remarks which helped us to improve this paper.

\providecommand{\bysame}{\leavevmode\hbox to3em{\hrulefill}\thinspace}
\providecommand{\MR}{\relax\ifhmode\unskip\space\fi MR }
\providecommand{\MRhref}[2]{%
  \href{http://www.ams.org/mathscinet-getitem?mr=#1}{#2}
}
\providecommand{\href}[2]{#2}

\end{document}